\documentclass[english,11pt]{article}

\usepackage{geometry}
\usepackage{amsmath,amsthm}
\usepackage{enumerate}
\usepackage{color}
\usepackage{amsfonts,amssymb}

\newcommand{\N}{\mathbb{N}}
\newcommand{\R}{\mathbb{R}}

\def\Var{{{\rm Var}}}

\newcommand{\E}{\mathbb{E}}

\def \({\left(}
\def \){\right)}
\def \[{\left[}
\def \]{\right]}

\newcommand{\eps}{\varepsilon}

\newcommand{\la}{\langle}
\newcommand{\ra}{\rangle}

\newcommand{\be}{\begin{equation}}
\newcommand{\ee}{\end{equation}}
\newcommand{\beqa}{\begin{eqnarray}}
\newcommand{\eeqa}{\end{eqnarray}}

\newcommand{\bea}{\begin{align}}
\newcommand{\eea}{\end{align}}

\DeclareMathAlphabet{\varmathbb}{U}{bbold}{m}{n}

\newcommand{\V}{\mathcal{V}}
\newcommand{\hamil}{\mathcal{H}_N}
\newcommand{\partf}{\mathcal{Z}_N}
\newcommand{\interv}{[-1,1]}

\newcommand{\norm}[1]{\left\| #1 \right\|}

\newcommand{\eqdist}{\overset{\rm d}{=}}
\newcommand{\convdist}{\xrightarrow{\rm d}}
\newcommand{\limN}{\xrightarrow{N\to+\infty}}
\newcommand{\bigo}[1]{\mathcal{O}(#1)}

\newcommand{\dyadic}{\mathcal{I}}

\newcommand{\noverlap}{R^{(k)}}

\newcommand{\proba}[1]{\mathbb{P}_\lambda #1}
\newcommand{\Edis}[1]{\mathbb{E} #1}
\newcommand{\Etherm}[1]{\big\langle #1 \big\rangle}
\newcommand{\Ethdis}[1]{\mathbb{E} \big\langle #1 \big\rangle}
\newcommand{\ETherm}[1]{\Big\langle #1 \Big\rangle}
\newcommand{\EThdis}[1]{\mathbb{E} \Big\langle #1 \Big\rangle}
\newcommand{\Elam}[1]{\mathbb{E}_{\lambda} #1 }

\newtheorem{theorem}{Theorem}
\theoremstyle{plain}
\newtheorem{proposition}{Proposition}
\newtheorem{lemma}{Lemma}
\newtheorem{corollary}{Corollary}
\theoremstyle{plain}

\theoremstyle{plain}
\newtheorem{definition}{Definition}
\theoremstyle{remark}

\theoremstyle{definition}

\begin{document}

\title{Strong replica symmetry for high-dimensional disordered log-concave Gibbs measures}
\author{Jean Barbier$^1$, Dmitry Panchenko$^2$, and Manuel S\'aenz$^1$}
\date{\footnotesize{$^1$\emph{The Abdus Salam International Center for Theoretical Physics, Trieste, Italy.}\\
$^2$\emph{Department of Mathematics, University of Toronto, Ontario, Canada.}}}

\maketitle

\begin{abstract}
    {We consider a generic class of log-concave, possibly random, (Gibbs) measures. We prove the concentration of an infinite family of order parameters called multioverlaps. Because they completely parametrise the quenched Gibbs measure of the system, this implies a simple representation of the asymptotic Gibbs measures, as well as the decoupling of the variables in a strong sense. These results may prove themselves useful in several contexts. In particular in machine learning and high-dimensional inference, log-concave measures appear in convex empirical risk minimisation, maximum a-posteriori inference or M-estimation. We believe that they may be applicable in establishing some type of ``replica symmetric formulas'' for the free energy, inference or generalisation error in such settings.}
    {Bayesian inference, disordered systems, multioverlap concentration, replica symmetry}
\end{abstract}

\section{Introduction}

The concentration of \emph{overlaps order parameters} is key to the theoretical understanding of plenty of high-dimensional statistical models. There are few canonical situations in which this phenomenon, often called ``replica symmetry'' in statistical mechanics \cite{mezard1987spin,mezard2009information}, is known to occur: high temperature regions in phase diagrams of disordered spin systems in statistical physics \cite{Talagrand2011spina,Talagrand2011spinb}, low constraint density regimes of random combinatorial optimisation problems in computer science \cite{talagrand2001high,montanari2006counting,panchenko2014replica,bapst2016harnessing,coja2018charting,coja2018replica}, the so-called ``Nishimori line'' of planted statistical mechanics systems, that is equivalent to optimal a-posteriori Bayesian inference \cite{nishimori01,contucci2009spin,barbier2019overlap,barbier2020strong}, ferromagnetic models \cite{chatterjee2015absence,barbier2019concentration}, and, finally, log-concave measures (see \cite{saumard2014log} and references therein). This paper is concerned with this last case. 

Under this replica symmetric regime, ``replica symmetric formulas'' can be proved for the free energy or mutual information of the problem. Two general schemes have been successfully developed to establish such formulas: a rigorous version of the \emph{cavity method} \cite{mezard1987spin} (also called the ``Aizenman-Sims-Starr scheme'' in physics \cite{aizenman2003extended}, or ``leave-one-out'' approach in statistics) which was applied to spin glass and inference problems \cite{aizenman2003extended,Talagrand2011spina,Talagrand2011spinb,panchenko2013sherrington,2016arXiv161103888L,lesieur2017statistical,coja2018information}, and the \emph{interpolation} \cite{guerra2002thermodynamic,guerra2003broken,Talagrand2011spina,Talagrand2011spinb} and \emph{adaptive interpolation \cite{barbier2019adaptivesimple,barbier2019adaptivereplica} methods}. The last one is specifically designed for analysing inference problems \cite{aubin2018committee,barbier2019optimal,barbier2019transition,barbier2020all,barbier2018mutual,barbier2017layered,gabrie2018entropy,barbier2018adaptive,barbier2019mutual,luneau2020information,reeves2020information,alberici2021multi,alberici2020solution}. Other, less standard, approaches exist as well \cite{giurgiu2016spatial,XXT,barbier2020mutual,reeves2016replica,el2018estimation}. All these methods rely direclty or implicitely on the concentration of the overlaps. Note, however, that in the context of inference problems, at the moment the concentration of overlaps is under control only in the \emph{Bayes optimal setting} (or ``matched teacher-student scenario''), corresponding to the \emph{Nishimori line} \cite{contucci2009spin,nishimori01}. Outside the Nishimori line (i.e., in models with mismatch) crucial identities used for these proofs are lost. Therefore, establishing the concentration of the overlaps in certain \emph{non Bayes optimal regimes of inference} is fundamental, and is the main goal of the present contribution.

Furthermore, controlling all \emph{multioverlaps} (instead of only the $2$-replicas overlap, which is a special case of multioverlap) is expected to be necessary in proofs of replica symmetric formulas for systems with a sparse quenched disorder \cite{de2004random,barra2007stability}. But controlling multioverlaps between more than two replicas, and thus statistics of the Gibbs measure of the system of an order higher than two, is highly non-trivial. In fact, there exist few situations in which some control of multioverlaps has been achieved: in certain dilute spin glasses \cite{guerra2004high}, constraint satisfaction problems \cite{talagrand2001high,montanari2006counting,panchenko2014replica}, ferromagnetic models \cite{barbier2019concentration} and Bayes optimal inference \cite{barbier2020strong}. Here we show that all the multioverlaps of the system concentrate for a very generic class of log-concave random (Gibbs) measures. That is, \emph{strong replica symmetry} holds. This implies concentration with respect to both the Gibbs measure and the disorder; this later concentration has been recently shown to \emph{not} hold for the $2$-replicas overlaps of some sparse model \cite{coja2021sparse}, so one does not imply the other. As a corollary we obtain a particularly simple representation of the (possibly non-unique) asymptotic Gibbs measures, as well as a strong decoupling of the variables with respect to both the Gibbs measures and the inherent quenched disorder of the system.

An important prior contribution in this direction and an inspiration to our work are the papers \cite{franz2003replica,de2009self}, where equivalent relations to the Ghirlanda-Guerra \cite{ghirlanda1998general,panchenko2010ghirlanda} and Aizenman-Contucci \cite{aizenman1998stability} identities were derived for the multioverlaps, partly heuristically. These works were latter extended in \cite{barra2007stability,sollich2012spin}. In this paper we prove, in the context of log-concave measures, a version of the results present in \cite{barbier2020strong} for Bayesian optimal inference. In particular, here we manage to establish the results for the original model whenever the Hamiltonian is uniformly strictly concave, while in \cite{barbier2020strong} this is done only for a perturbed one. Also, the settings of both works are very different and the proof techniques are distinct in many important ways. Another relevant prior work is the analysis of the Shcherbina-Tirozzi model in Chapter 3 of \cite{Talagrand2011spina}, where properties of log-concave measures were used to prove the concentration of the order parameters, see also \cite{barbier2021performance}. However, their results differ from ours in that they only show this for the canonical $2$-replicas overlap and, also, in that they need the disorder measure to be log-concave, while this is not necessary for the proofs in this paper.

We expect our results to be useful in the study of non Bayes optimal log-concave inference problems and in general statistical mechanics models with sparse quenched disorder. One example would be Bayesian linear regression with mismatch. In \cite{barbier2021performance}, one instance of this problem is analysed by means of concentrations deriving from log-concavity (although, instead of using the results presented here, some similar tailor-made concentration bounds are derived for that particular instance). Another important example of a possible application would be compressed sensing with $\ell_1$ penalty and involving a sparse sensing matrix \cite{angelini2012compressed}, as in this case, the control of all multioverlaps would be necessary. Finally, this work might also have an impact on (robust) statistics \cite{huber2009robust} and convex optimisation \cite{boyd2004convex}.

\paragraph{Organisation of the paper.} The remaining of the paper is organised as follows. In Section \ref{sec:settings} we introduce the setting in which our results apply and the notations used throughout the paper along with our main results. These results are a direct consequence of the intermediate ones present in Section \ref{sec:intermediate}, that rely on the introduction of a perturbation of the system (this perturbation is an important tool for our proof). After that, in Section \ref{sec:technical}, we give a brief account on the necessary technical background needed for the proofs. Finally, in Section \ref{sec:proofs} the complete proofs are provided for the perturbed system. Then at the end of it we show that the perturbation can be ``removed'' to get back the main results stated in Section \ref{sec:settings} for the original (non perturbed) model. In the Appendix, some complementary proofs are included for the convenience of the reader.

\section{Setting and main results}\label{sec:settings}

Let $J$ be some collection of real valued random variables, which may be taken to be a vector, matrix, tensor, etc. Consider a bounded random vector $\sigma := (\sigma_1,\dots,\sigma_N) \in \Sigma_N := \interv^N$ distributed, conditionally on $J$, according to the Borel probability measure on the set $\Sigma_N$ defined by the probability density
\begin{equation}\label{eq:gibbs}
    G_N(\sigma\mid J) := \frac{1}{\partf(J)} \exp \hamil(\sigma \mid J) \quad \mbox{with}\quad \partf(J)=\int_{\Sigma_N} d\sigma \exp \hamil(\sigma \mid J)\,,
\end{equation}
where $\hamil$ is, for almost every $J$, some twice differentiable concave function of $\sigma$ and $\partf(J)$ is a normalisation constant. As we will discuss in further detail in Section \ref{sec:technical}, expression \eqref{eq:gibbs} defines a log-concave distribution. We can assume that the variables $\sigma_i$ belong to $\interv$ without loss of generality (w.l.g.) as long as they are bounded uniformly in $N$.

In this paper we use statistical mechanics nomenclature and notations. For instance, we will call the coordinates of the random vector $\sigma$ as \emph{spins}. The log-concave distribution \eqref{eq:gibbs} resulting from conditioning with respect to (w.r.t.) the disorder $J$ will be referred as the \emph{Gibbs distribution}, the concave function $\hamil$ that defines it as the \emph{Hamiltonian}, and $\partf(J)$ as the \emph{partition function}. Likewise, we will refer to conditionally independent and identically distributed samples from the Gibbs distribution corresponding to the same realisation of the disorder $J$ as \emph{replicas}.

As mentioned above, from now on and throughout the rest of the paper we will assume that the Hamiltonian is, for almost every realisation of $J$ (w.r.t. the disorder measure), twice differentiable (w.r.t. the spin variables) and concave. This implies that, for almost every $J$, its Hessian $H[\hamil]$ is a negative semi-definite matrix in every point of $\Sigma_N$. That is, for every vector $v \in \R^N$ and $\sigma \in \Sigma_N$ we have that
\begin{equation}\label{eq:concave}
    \big(v,H[\hamil](\sigma)v\big) = \sum_{i,j=1}^N v_i v_j \frac{d^2\hamil}{d\sigma_id\sigma_j}(\sigma)  \leq 0\,,
\end{equation}
where $(\cdot,\cdot)$ is the usual inner product of $\R^N$. We will also assume that the random Hamiltonian (due to the randomness $J$) is symmetric under exchanges of spin indices; that is, we have that for every permutation $P: [N] \to [N]$ of $[N]:=\{1,\ldots,N\}$, it holds that
\begin{equation}\label{eq:spin_symm}
    \hamil(\sigma_1,\dots,\sigma_N\mid J) \eqdist \hamil(\sigma_{P(1)},\dots,\sigma_{P(N)}\mid J)\,,
\end{equation}
where $\eqdist$ stands for equality in distribution.

We define the \emph{free entropy} as the log-partition function $F_N= F_N(J) := \ln\partf(J)$. Likewise, the \emph{mean free entropy} is
\begin{equation}
    \Edis{F_{N}}  := \E \ln  \int_{\Sigma_N} d\sigma \exp{\hamil(\sigma \mid J)}\,.
\end{equation}
For simplicity, and because we will mostly be concerned with the mean free entropy, when there is no ambiguity we will simply refer to $\E F_N$ as the free entropy of the system.

\paragraph{Gaussian regularisation.} Here we introduce what we call the \emph{Gaussian regularisation} (corresponding to a ``ridge regularisation'' in machine learning). Its role will be to ensure the concentration w.r.t. the regularised Gibbs measure (or ``thermal'' concentration) of generic order parameters coined multioverlaps; these are natural generalisations of the usual magnetisation and Edwards-Anderson overlap in spin glasses.

This regularisation added to the Hamiltonian is defined according to
\begin{equation}\label{eq:gauss_pert}
    \hamil^{\rm gauss}(\sigma) := - \frac{\eps_N}{2} \norm{\sigma}^2
\end{equation}
where $\norm{\,\cdot\,}$ is just the $L^2$ norm of $\R^N$. Here the regularisation strenght $\eps_N$ is such that 
\begin{equation*}
    1 \geq \eps_N \to 0 \ \ \ \ \mbox{and} \ \ \ \  N\eps_N \to +\infty\,.  
\end{equation*}
As we will see later, the first condition implies that the regularisation is $o(N)$ and therefore sub-dominant w.r.t. the original Hamiltonian, while the second condition implies the thermal concentration of the multioverlaps via the Brascamp-Lieb inequality.

Let us make a remark concerning the fact that the Gaussian regularisation is \emph{not} always required. As we will see later, the Gaussian regularisation is needed to ensure that, in every point of $\Sigma_N$, the Hessian of the Hamiltonian is upper bounded (in the sense of Loewner partial order) by $-\eps_N \mathbb{I}$. In the cases where this bound can be proved for the original Hamiltonian, this regularisation can be omitted. However, we include it in our results so that they apply to the general concave case.

From now on, for $i,l\geq1$, we use the notation $\sigma_i^l$ to refer to the $i$-th spin of the $l$-th replica (i.e., sample) of the regularised system with Hamiltonian $\hamil(\cdot \mid J)+\hamil^{\rm gauss}(\cdot)$. Whenever the replica index $l$ is omitted in this notation, it is understood that the replica in question is the first one:
$$\sigma:=\sigma^1\,.$$

We will also define the mean w.r.t. the Gibbs measure of the regularised model as $\Etherm{\,\cdot\,}_0$; namely, for any function $A$ of multiple replicas from the regularised model $(\sigma^l)_{l\in \mathcal{C}}$, with $\mathcal{C}$ some finite set of integers, we have
\begin{align}
 \big\langle A((\sigma^{l})_{l\in \mathcal{C}})\big\rangle_0 &:= \frac{1}{\tilde{\mathcal{Z}}_N^{|\mathcal{C}|}} \int_{\Sigma_N^{|\mathcal{C}|}} A((\sigma^{l})_{l\in \mathcal{C}}) \prod_{l \in \mathcal{C}} d\sigma^l\exp\big(\hamil(\sigma^{l} \mid J)+\hamil^{\rm gauss}(\sigma^{l})\big) \,,\label{bracket_soft_orig}
\end{align}
where $\tilde{\mathcal{Z}}_N$ is the partition function associated to the original Hamiltonian plus the Gaussian regularisation (if needed).

\paragraph{Perturbed free entropy variance.} Let $(s_N)_{N\geq1}$ be some sequence of non-negative real numbers s.t. $s_N \to +\infty$ and $s_N/N \to 0_+$ as $N\to+\infty$. Denote by $\V_N$ the set of real continuous functions on $\Sigma_N$ whose absolute value is uniformly bounded by $s_N$. 

We will define the perturbed free entropy of the model according to 
\begin{equation}
    {F^{(V)}_N}  := {\ln  \int_{\Sigma_N} d\sigma\exp{\big(\hamil(\sigma \mid J)+\hamil^{\rm gauss}(\sigma)+V(\sigma)\big)}}\,.
\end{equation}
Then, the perturbed free entropy variance will be given by 
\begin{equation}\label{eq:def_v_N}
    v_N := \sup_{V \in \mathcal{V}_N} \E\big[ (F^{(V)}_N - \E F^{(V)}_N )^2 \big].
\end{equation}
In typical situations, $\Var F_N$ can be proved to be $\bigo{N}$, which means that $F_N/N$ concentrates. For our results, we will assume that $v_N$ scales in this same manner; that is, $v_N = \bigo{N}$. This is expected to be the case in most situations where the original free entropy concentrates, as the variance of the free entropy is determined by the randomness in the model and the $V(\sigma)$ are non-random. In Section \ref{sec:intermediate} we will, at the expense of introducing further notation, slightly relax this condition of perturbed free energy concentration.

\paragraph{Multioverlaps.} Our main object of interest is a family of order parameters coined \emph{multioverlaps}. These fully encode the quenched Gibbs measure of the replicated system, as the joint moments $\E\langle \prod_{(i,l)\in \mathcal{C}} \sigma_{i}^{l}\rangle$ can be re-expressed as functions of multioverlaps and vice-versa \cite{barbier2020strong}. Special cases of multioverlap are the magnetisation $N^{-1} \sum_{i=1}^N \sigma_i$ as well as the usual 2-replica overlap (or Edwards-Anderson order parameter):
\begin{equation}
    R := \frac{1}{N} \sum_{i=1}^N \sigma_i^1 \sigma_i^2\,.
\end{equation}
For each $n\geq1$, define the sets $K_n := \mathbb{N}^n$ and $K := \bigcup_{n\geq1} K_n$. Similarly to the overlap, given $n\geq1$ and $k \in K_n$, we define the associated \emph{multioverlap} by
\begin{equation}
    \noverlap := \frac{1}{N} \sum_{i=1}^N \prod_{l=1}^n (\sigma_i^{l})^{k_l}\,.
\end{equation}
Note that in the definition of multioverlaps, we assume that the replicas involved are the first $n$ ones. But because replicas are exchangeable, they may be taken to be any replica set without affecting any computation.


\subsection{Main results}\label{sec:results}

Our first result establishes the thermal concentration of all the multioverlaps $(R^{(k)})_{k\in K}$ and is a consequence of Brascamp-Lieb's variance inequality.
\begin{theorem}[Thermal multioverlap concentration]\label{prop:mo_thm_conc}
    Assume that the Hessian of $\hamil$ exists and is a negative semi-definite matrix for every $\sigma\in\Sigma_N$. Then, for all $n\geq 1$ and $k \in K_n$,
    \begin{equation*}
        \Big\langle \big(\noverlap - \big\langle{\noverlap\big\rangle_0\big)^2 \Big\rangle_0 \leq \frac{\|k\|^2}{N \eps_N}}\,.
    \end{equation*}
\end{theorem}

As a consequence we have the following asymptotic thermal decorrelation result.
\begin{corollary}[Asymptotic thermal decorrelation]\label{cor:asymp_indep_thm}
    Under the hypotheses of Theorem~\ref{prop:mo_thm_conc} and assuming that the regularisation $N \varepsilon_N \to +\infty$, for every $k\geq 1$, any collection $h_1,\dots,h_k:[-1,1] \to \R$ of continuous functions, and any set of distinct spin indices $(i_{j})_{j\le k}$ we have that 
    \begin{equation*}
        \Big\langle {\prod_{j=1}^k h_j(\sigma_{i_j})}\Big \rangle_0 - \prod_{j=1}^k \Big\langle {h_j(\sigma_{i_j})}\Big \rangle_0 \xrightarrow{\mathbb{P}} 0\,.
    \end{equation*}
\end{corollary}

Our main result is a much stronger concentration theorem, namely self-averaging w.r.t. the quenched Gibbs measure. 

\begin{theorem}[Strong replica symmetry]\label{thm:mo_original_conc}
    Suppose that the Hessian of $\hamil$ exists and is a negative semi-definite matrix for every $\sigma\in\Sigma_N$, that the exchangeability of spins \eqref{eq:spin_symm} holds, and that there exists a sequence $(s_N)_{N\geq1}$ s.t. $s_N \to +\infty$, $s_N/N \to 0_+$ and s.t. the free entropy variance verifies $v_N^{1/2}/s_N \to 0_+$. Then, for all $k\in K$ we have that
\begin{equation*}
    \mathbb{E}{\Big\langle \big(\noverlap - \mathbb{E}\big\langle\noverlap\big\rangle_0\big)^2 \Big\rangle}_0 \xrightarrow{N\to+\infty} 0\,.
\end{equation*}
\end{theorem}

Note that the assumption that $v_N^{1/2}/s_N$ must be a vanishing sequence implicitly requires that, for every $V\in\V_N$, the perturbed free energy $F^{(V)}_N$ does not deviate too much from its mean $\E F^{(V)}_N$. The higher its variance $v_N$, the faster the sequence $s_N$ must grow (under the constraint $s_N = o(N)$). Therefore this (weak) assumption restricts the class of models to ``well-behaved'' ones (which is a very large set of relevant statistical models, including those mentionned in the introduction).

\paragraph{Remark concerning planted models.} In the context of Bayesian inference models (also called ``planted models'' in spin glass literature), the disorder $J$ usually consists in part in some ``ground-truth'' vector $\sigma^*$ that is to be reconstructed, and on which the data available to the statistician depends (we may assume that $\sigma^*$ belongs to $\Sigma_N$). In this context it may be interesting to also prove the concentration of the  correlation $m^*:=N^{-1} \sum_{i=1}^N \sigma^*_i \sigma_i$ between the vector $\sigma$, interpreted as an estimator, and the ground-truth. This correlation is a measure of how strongly the estimator $\sigma$ ``magnetises'' onto the hidden $\sigma^*$, and is therefore a natural way to quantify the quality of inference. Note that a similar Gaussian regularisation as the one introduced in the previous section can be used to enforce the thermal concentration of $m^*$. Define $\tilde{\mathcal{H}}^{\rm gauss}_N(\sigma) := - (\tilde \eps_N/2) \sum_{i=1}^N (\sigma_i^* \sigma_i)^2$, where $\tilde \eps_N$ is another sequence such that $\tilde \eps_N \to 0_+$ and $N\tilde\eps_N \to +\infty$. Take the regularised Hamiltonian of  the system to be given by $\hamil + \mathcal{H}^{\rm gauss}_N + \tilde{\mathcal{H}}^{\rm gauss}_N$. Then, Theorem \ref{thm:mo_original_conc} still holds adding this second Gaussian regularisation. Furthermore, if we assume that (for every $i\in[N]$) $\mathbb{P}(\sigma_i^*=0)=0$, we can make the change of variables $\tilde \sigma_i := \sigma_i^* \sigma_i$. Thus, Theorem \ref{thm:mo_original_conc} still holds for these new variables as the Hamiltonian with the new variables is a concave and exchange-symmetric function of $\tilde \sigma$. Then, we have that the associated multioverlaps of $\tilde \sigma$ also concentrate; in particular, so does $m^*$. Therefore, we have that the multioverlaps $(R^{(k)})_{k\in K}$ concentrate along with $m^*$.\\

Theorem \ref{thm:mo_original_conc} implies a simple representation for the asymptotic distribution of the spins.
\begin{corollary}[Asymptotic spin distribution]\label{cor:asymp_spin_dist}
    Under the hypotheses of Theorem \ref{thm:mo_original_conc}, for every subsequence $(N_j)_{j\geq1}$ s.t. the replicated system $(\sigma_i^l)_{i,l\geq1}$  converges in distribution along it, there exists a (non random) probability measure $\nu(\cdot)$ over the set of Borel probability measures on $\interv$ such that, for all $i\geq1$, the spin variables $(\sigma_i^l)_{l\geq1}$ converge jointly in distribution towards independent samples from $\mu_i(\cdot)$; where the $(\mu_i)_{i\ge 1}$ are i.i.d. random measures distributed according to $\nu$.
\end{corollary}
Note that the measure $\nu(\cdot)$ is not necessarily unique and may a priori depend on the subsequence $(N_j)_{j\geq1}$. We also obtain as a further consequence the following strong asymptotic independence.  
\begin{corollary}[Strong asymptotic spin independence]\label{cor:asymp_indep_disthm}
    Under the hypotheses of Theorem \ref{thm:mo_original_conc}, for every $k\geq 1$ and any collection $h_1,\dots,h_k:[-1,1] \to \R$ of continuous functions, we have that
    \begin{equation*}
        \E \Big\langle {\prod_{j=1}^k h_j(\sigma_{j})}\Big \rangle_0 - \prod_{j=1}^k \E \Big\langle {h_j(\sigma_{j})}\Big\rangle_0 \xrightarrow{N\to+\infty} 0\,.
    \end{equation*}
\end{corollary}
In the above statement, by spin permutation symmetry (hypothesis \eqref{eq:spin_symm}), we fix w.l.g. the spin indices to be the first $k$ ones; but this choice is arbitrary.


\section{Intermediate results for the perturbed model}\label{sec:intermediate}

Instead of working directly with the original model \eqref{eq:gibbs}, we will study a slightly \emph{perturbed} one. Only after proving strong replica symmetry for this perturbed model, we will extend the results to the original (possibly regularised) one. The objective of this perturbation is to make the system have ``good structural properties'' that ensure concentrations necessary for our results. While the perturbation present in this work is new, the general strategy is ubiquitous in statistical mechanics. As we will see later, if the scaling of the perturbation is chosen appropriately, it will ensure the convergences needed while not changing the asymptotic values of the multioverlaps. Thus, we will be able to show that the concentrations also hold for the original (unperturbed) model.

We will refer to this perturbation as the \emph{Poisson perturbation}. Its role will be to force some Franz-de Sanctis type inequalities \cite{de2009self}, which are generalisations of the Ghirlanda-Guerra identities in spin glasses \cite{ghirlanda1998general,panchenko2010ghirlanda}. These inequalities will be at the core of our proof of multioverlap concentration.

\paragraph{Poisson perturbation.} For every $p\geq0$, define the set $\mathcal{D}_p := \{ 2^{-k} \in [0,1] : k \in \mathbb{N}_{\ge p} \}$ as well as $\mathcal{I}_m := (\mathcal{D}_0 \times \mathcal{D}_1 \times \cdots \times \mathcal{D}_{m-1})$ for every $m\geq1$ (with $\times$ being the Cartesian product), and $\mathcal{I} := \bigcup_{m\ge 1} \mathcal{I}_m$ (a countable set). Let $\iota:\mathcal{I} \to \N$ be a function s.t. for $I :=( 2^{-i_0},\dots,2^{-i_{m-1}})\in\mathcal{I}$ it gives $\iota(I) = \sum_{p=0}^{m-1} i_p$. For each $I=(a_0,\dots,a_{m-1})\in\mathcal{I}$, consider the polynomial $P_I : \interv \to [0,1]$ given, for every $x\in\interv$, by
\begin{equation}\label{eq:def_poly}
    P_I(x):= 2^{-\iota(I)-2m} \sum_{p=0}^{m-1} a_p (x+1)^p\,.
\end{equation}
The definitions of $\mathcal{I}$ and the polynomials were chosen so that the expression that will give the Poisson perturbation is summable and (for every $I\in\dyadic$ and $x\in\interv$) $a_p (x+1)^p$ belongs to $[0, 1]$. Observe that all these polynomials are convex on the interval $\interv$. Furthermore, let $\pi := (\pi_{I})_{I\in\mathcal{I}}$ be a collection of i.i.d. Poisson random variables of mean $s_N$, $U := (U^I_j)_{j\ge 1,I\in\mathcal{I}}$ be a collection of i.i.d. uniform random variables in $[N]$, and $\lambda := (\lambda_{I})_{I\in\mathcal{I}}$ another collection of i.i.d. random variables uniform in $[1/2,1]$. All these random variables are independent of everything else. Here, similarly to the case of the Gaussian regularisation, $(s_N)_{N\geq1}$ is some sequence verifying
\begin{equation*}
    s_N \to +\infty \ \ \ \ \mbox{ and } \ \ \ \  1 \geq \frac{s_N}{N} \to 0\,.
\end{equation*}
Then, the Poisson perturbation of the Hamiltonian will be
\begin{equation}\label{eq:poiss_pert}
    \hamil^{\rm poiss}(\sigma\mid \pi,U,\lambda):= - t \sum_{I\in \mathcal{I}}  \lambda_I \sum_{j=1}^{\pi_I} P_I(\sigma_{U^I_j})\,,
\end{equation}
where $t\in[0,1]$ is a parameter that tunes the strength of the perturbation. Note that if $(\pi_{I,i})_{I\in\mathcal{I},i\in[N]}$ are i.i.d. random variables with Poisson distribution of mean $s_N/N$ (instead of $s_N$), so clearly $$\sum_{i=1}^{N} \pi_{I,i} P_I(\sigma_{i}) \eqdist \sum_{j=1}^{\pi_I} P_I(\sigma_{U^I_j}).$$ We can then define the perturbation alternatively as
\begin{equation}\label{eq:poiss_pert2}
     \hamil^{\rm poiss}(\sigma\mid \pi,U,\lambda):= - t \sum_{I\in \mathcal{I}}  \lambda_I \sum_{i=1}^{N} \pi_{I,i} P_I(\sigma_{i})\,.
\end{equation}
We will use these two alternative forms of the Poisson perturbation Hamiltonian interchangeably. The summability of \eqref{eq:poiss_pert} can be seen from the fact that (for $\delta < 1$) the tails of the convergent series $\sum_{I\in\mathcal{I}} 2^{-(1-\delta)\iota(I)}$ are eventually a.s. larger than those of the Poisson perturbation.

All together, \eqref{eq:gauss_pert} and \eqref{eq:poiss_pert} result in a random perturbed Hamiltonian given by 
\begin{equation}\label{pert_mode}
\hamil'(\sigma)=\hamil'(\sigma \mid J,\pi,U,\lambda):= \hamil(\sigma \mid J)+ \hamil^{\rm gauss}(\sigma) + \hamil^{\rm poiss}(\sigma\mid \pi,U,\lambda)\,,
\end{equation}
where, when needed, we explicitly emphasize the dependence of this Hamiltonian on the disorder $(J,\pi,U)$ and the perturbation parameters $\lambda$. Note that it remains true for the random perturbed Hamiltonian that for every permutation of spin indices $P: [N] \to [N]$,
\begin{equation}
    \hamil'(\sigma_1,\dots,\sigma_N\mid J,\pi,U,\lambda) \eqdist \hamil'(\sigma_{P(1)},\dots,\sigma_{P(N)}\mid J,\pi,U,\lambda)\,.\label{exch_pert}
\end{equation}

The Gibbs measure of the perturbed model will be defined in an analogous way to the unperturbed one. That is, the spin system in the perturbed model will be distributed, conditional on the disorder $(J,\pi,U)$ and the regularisation $\lambda$, according to the Borel probability measure $\proba{(\cdot\mid J,\pi,U)}$ on the set $\Sigma_N$ defined by the probability density 
\begin{align}\label{eq:gibbs_pert}
    G'_N(\sigma \mid J,\pi,U,\lambda) := \frac{\exp{\hamil'(\sigma \mid J,\pi,U,\lambda)}}{\partf'(J,\pi,U,\lambda)}\,, \ \ \partf'(J,\pi,U,\lambda)=\int_{\Sigma_N}d\sigma \exp{\hamil'(\sigma \mid J,\pi,U,\lambda)}\,.
\end{align}

The thermal mean for this measure is also defined in an analogous way. Whenever we want to explicitly emphasise the dependence of this mean on the strength parameter $t$, we will write $\Etherm{\, \cdot \,}_t$. When not done, it is implicitly assumed that the $t$ is taken to be $1$: $\Etherm{\, \cdot \,}=\Etherm{\, \cdot \,}_{t=1}$ Also, note that $\Etherm{\, \cdot \,}_0 = \Etherm{\, \cdot \,}_{t=0}$ (i.e., the Gibbs measure for $t=0$ is equal to the measure without the Poisson perturbation). From now on, the expectation $\E(\cdot)$ will be taken to be over the extended disorder $(J,\pi,U)$. For the expectation over the regularisation $\lambda$ we will always write $\E_\lambda(\cdot)$ explicitly.

\paragraph{Poisson perturbed free entropy variance.} The free entropy of the perturbed model is
\begin{equation}\label{Free_en_pert}
    {F'_{N}}  := {\ln  \int_{\Sigma_N} d\sigma\exp{\hamil'(\sigma \mid J,\pi,U,\lambda)}}\,.
\end{equation}
Note that the free entropy defined in this way is a function of the regularisations $\lambda$ and $\eps_N$ (in addition of $(J,\pi,U)$). It is easy to see that, because the regularisations added to the Hamiltonian are of order $o(N)$, the free energies of the perturbed and unperturbed systems, both of order $N$, are equal up to a lower order $o(N)$ correction (see for example \cite{barbier2020strong}). Finally define 
\begin{equation}\label{eq:def_v_N_pois}
    v'_N := \sup_{\lambda \in [1/2,1]^\N} \E\big[ (F'_N - \E_J F'_N )^2 \big]\,,
\end{equation}
where $\E_J( \cdot )$ is the expectation only with respect to the disorder $J$ (i.e., the variables $\pi$ are fixed). For these intermediate results, we will assume that $v'_N$ is $\bigo{N}$.

\paragraph{Intermediate results.} The main intermediate result concerns the full concentration of multioverlaps.

\begin{proposition}[Strong replica symmetry for the perturbed model]\label{prop:mo_conc}
    Suppose that the Hessian of $\hamil$ exists and is a negative semi-definite matrix for every $\sigma\in\Sigma_N$, that the exchangeability of spins \eqref{eq:spin_symm} holds, and that there exists a sequence $(s_N)_{N\geq1}$ s.t. $s_N \to +\infty$, $s_N/N \to 0_+$, and for which $(v'_N)^{1/2}/s_N\to 0_+$. Then, for all $n\geq 1$ and $k \in K_n$, we have that
    \begin{equation*}
       \Elam{\EThdis{\big(\noverlap - \mathbb{E}\big\langle \noverlap\big\rangle\big)^2}} \limN 0\,.
    \end{equation*}
\end{proposition}

Compared to the (stronger) Theorem~\ref{thm:mo_original_conc}, the above result is valid in average over the perturbation parameters $\lambda$. This disorder concentration can then be extended to the original model by means of the following proposition.

\begin{proposition}[Asymptotic independence of multioverlaps on the perturbation]\label{prop:thm_mean_dif}
    Under the hypotheses of Theorem \ref{prop:mo_conc}, for every $k\in K$ there exists a constant $C(k) > 0$ s.t.
    \begin{equation*}
        \E \Big| \Etherm{R^{(k)}} - \Etherm{R^{(k)}}_0 \Big| \leq C(k) \Big(\frac{s_N}{N}\Big)^{1/6}(1+o_N(1))\,.
    \end{equation*}
\end{proposition}

Because these results are established under the slightly more relaxed condition of having control over $v'_N$ (instead of $v_N)$, the main results of Section \ref{sec:settings} also hold when replacing $v_N$ by $v'_N$ in their statements.

\section{Technical background}\label{sec:technical}

In this section we will present a brief summary of the results on log-concave distributions and the Aldous-Hoover representation that are needed for the proofs. 

\subsection{Log-concave distributions.} 
We begin by presenting the broader notion of log-concave measure.

\begin{definition}[Log-concave measure]
    A Borel probability measure $\mathbb{P}{(\cdot)}$ on some convex set $C \subseteq \R^N$ is said to be \emph{log-concave} if for all non-empty measurable sets $A,B\in C$ and for all $0 < t <1$,
    $$\mathbb{P}{(t A + (1-t) B)} \geq \mathbb{P}{(A)}^t \mathbb{P}{(B)}^{(1-t)}\,.$$
\end{definition}

An important property of log-concave measures is that they have sub-exponential tails \cite{borell1984convexity}. Another fundamental result is that they obey Paouris' inequality \cite{Paouris,adamczak2014short} that relates strong with weak moments. Within this class of measures, there are the ones defined by log-concave densities. A thorough review on log-concave measures may be found in \cite{saumard2014log}.

\begin{definition}[Log-concave density]
    An absolutely continuous probability measure $\mathbb{P}(\cdot)$ on some convex set $C \subseteq \R^N$ is said to have a \emph{log-concave density} $f(\cdot)$ if $f = e^{\phi}$, for some $\phi: \R^N \to \R$ concave.
\end{definition}
Under some regularity conditions for $\phi(\cdot)$, Brascamp-Lieb's inequality \cite{brascamp2002extensions} bounds the variance of functions of random vectors sampled according to the log-concave distribution induced by $\phi(\cdot)$.

\begin{theorem}[Brascamp-Lieb's inequality]\label{thm:brascamp}
    Let $A \in \R^N$ be a convex set and $\phi \in \mathcal{C}^2(A)$ be a strictly concave function. Then if the random variable $X$ is distributed according to the log-concave probability measure induced by $\phi(\cdot)$ and $g \in \mathcal{C}^1(A)$, we have that
    \begin{equation*}
        \Var{ g(X) } \leq - \E{ [\nabla g(X)^\intercal H^{-1}[\phi](X) \nabla g(X)] }\,.
    \end{equation*}
\end{theorem}

From this fundamental inequality we get the following corollary.

\begin{corollary}[Simplified Brascamp-Lieb's inequality]\label{cor:brascampstrong}
    Under the hypotheses of Theorem \ref{thm:brascamp} and if furthermore the Hessian $H[\phi]$ is uniformly upper bounded in $A$ by $- \eps \mathbb{I}$ for some $\eps > 0$, then 
    \begin{equation*}
        \Var{ g(X) } \leq \frac{1}{\eps} \E{ \norm{\nabla g(X)}^2 }.
    \end{equation*}
\end{corollary}
Here we understand that a linear operator $A$ of $\R^N$ upper bounds another one $B$ if $A-B$ is a non-negative linear operator.


\subsection{Aldous-Hoover representation.}\label{SecAHsoft1}
Here we will give the necessary background on a convenient and powerful representation of the asymptotic spin distribution that will be used along the proofs. 

\begin{definition}
    Let $(X_{ij})_{i,j\geq1}$ be an array of real-valued random variables. This array will be said to be \emph{separately exchangeable} if for every pair of bijections $P,P':\N \to \N$ we have that
    \begin{equation*}
        X_{i,j} \eqdist X_{P(i),P'(j)}\,.
    \end{equation*}
\end{definition}

The Aldous-Hoover theorem, proven independently in \cite{aldous1981representations} and \cite{hoover1982row}, gives a necessary and sufficient condition for a $2$-index array of random variables to be separately exchangeable. In particular, it gives a useful representation for working with this kind of arrays.

\begin{theorem}[Aldous-Hoover representation]\label{thm:AHrep}
    An array of real-valued random variables $(X_{i,j})_{i,j\geq1}$ is separately exchangeable if and only if for every $i,j\geq1$
    \begin{equation*}
        X_{i,j} \eqdist f(u,v_i,w_j,x_{i,j})\,,
    \end{equation*}
    for some measurable function $f:[0,1]^4 \to \R$ and (for every $i,j\geq1$) $u,v_i,w_j,x_{i,j}$ are i.i.d. uniform random variables on $[0,1]$.
\end{theorem}

Note that because the spins are bounded, each individual spin defines a tight sequence of random variables (w.r.t. the sequence index given by the size of the system). This means that for every $i,l\geq1$, there exists some subsequence of system sizes $(N^{(i,l)}_j)_{j\geq1}$ such that along this subsequence $\sigma_i^l$ converges weakly to $\gamma_i^l$ as $j\to+\infty$, for some random variable $\gamma^l_i$ taking values in $\interv$. By a diagonal argument, one may then take a common subsequence of system sizes $(N_j)_{j\ge 1}$ such that for every spin index $i\geq1$ and replica index $l\geq1$, it holds that $\sigma_i^l$ converges to $\gamma_i^l$.

By property \eqref{exch_pert} of the perturbed Hamiltonian and because replica indices are exchangeable, we have that the array of random variables $(\gamma_i^l)_{i,l\geq1}$ is separately exchangeable. Then, by Theorem~\ref{thm:AHrep} we know that there exists some measurable function $\sigma:[0,1]^4 \to \interv$ such that along this subsequence, for every $i,l\geq1$, we have that
\begin{equation*}
    \sigma_i^l \convdist \sigma(u,v_i,w_l,x_{i,l})\,,
\end{equation*}
where (for every $i,l\geq1$) $u,v_i,w_l,x_{i,l}$ are i.i.d. uniform random variables on $[0,1]$. Let us also define, for every $k\geq 1$, the associated functions $\bar \sigma^{(k)}: [0,1]^3 \to \interv$, that are the asymptotic equivalent of ``generalised magnetisations'', according to 
\begin{equation}
    \bar \sigma^{(k)}(u,v,w):= \int_0^1 \sigma^k(u,v,w,x) dx\,.    
\end{equation}

\begin{lemma}[Asymptotic multioverlap representation]\label{lem:mo_limit1}
    Along the subsequence $(N_j)_{j\ge 1}$ defined above, we have that for every $n\geq1$ and $k \in K_n$,
    \begin{equation*}
        \noverlap \convdist R^{(k)}_\infty(u,w_1,\dots,w_n) := \int_0^1 \prod_{l=1}^n \bar{\sigma}^{(k_l)}(u,v,w_l) \, dv\,.
    \end{equation*}
\end{lemma}
\begin{proof}
    Given $k:=(k_1,\dots,k_n)\in K$, by definition the associated multioverlap is
    \begin{equation*}
        R^{(k)} = \frac{1}{N} \sum_{i=1}^N \prod_{l=1}^n (\sigma_i^{l})^{k_l}\,.
    \end{equation*}
    Then, if all the spin variables converge in distribution towards their corresponding Aldous-Hoover limit as described in Section \ref{sec:technical}, we have that
    \begin{equation*}
        R^{(k)} - \frac{1}{N} \sum_{i=1}^N \prod_{l=1}^n \sigma^{k_l}(u,v_i,w_{l},x_{i,l}) \convdist 0\,.
    \end{equation*}
    For almost every $u\in[0,1]$ and $(w_{l})_{l\le n}\in[0,1]^n$ fixed, we have that the variables $\prod_{l=1}^n \sigma^{k_l}(u,v_i,w_{l},x_{i,l})$ are i.i.d. of mean
    \begin{equation*}
        \int_0^1\cdots\int_0^1 \Big(\prod_{l=1}^n \sigma^{k_l}(u,v_i,w_{l},x_{i,l}) \, dx_{i,l}\Big) dv_i\,.
    \end{equation*}
    This, in the notation of Section \ref{sec:technical}, is equal to $\int_0^1 \prod_{l=1}^n \bar \sigma^{(k_l)} dv$. Then, by the law of large numbers
    \begin{equation*}
        \frac{1}{N} \sum_{i=1}^N \prod_{l=1}^n \sigma^{k_l}(u,v_i,w_{l},x_{i,l}) \convdist \int_0^1 \prod_{l=1}^n \bar \sigma^{(k_l)}(u,v_,w_{l}) dv\,,
    \end{equation*}
    from which the result directly follows.
\end{proof}

The following lemma gives an analogous result for the thermal mean of the multioverlaps. Its proof is also very similar.
\begin{lemma}[Asymptotic mean multioverlap representation]\label{lem:mo_limit2}
    Along the subsequence $(N_j)_{j\ge 1}$ defined above, we have that for every $n\geq1$ and $k \in K_n$,
    \begin{equation*}
        \big\la{R^{(k)}}\big\ra \convdist  \int_0^1 \cdots \int_0^1 R^{(k)}_\infty(u,w_1,\dots,w_n) \, dw_1 \dots dw_n := \big\la{R^{(k)}_\infty}\big\ra(u)\,.
    \end{equation*}
\end{lemma}
\begin{proof}
    Let $k$ and $R^{(k)}$ be as before. The reasoning in the proof of this lemma is similar to the previous one. The difference is that, because here we also have a thermal mean, we need to approximate it by an empirical mean over a sufficiently large number of i.i.d. draws from the Gibbs measure; namely, of replicas. Fix a sequence $(t_N)_{N\geq1}$ and define (for $m\geq 0$) $l^{(m)} :=(1+m n,2+m n,\dots,n+nm)\in \mathbb{N}^n$. If $t_N$ is chosen so that it grows sufficiently fast then 
    \begin{equation*}
        \Etherm{R^{(k)}} - \frac{1}{t_N} \sum_{m=0}^{t_N-1} R^{({l^{(m)}},k)} \convdist 0\,,
    \end{equation*}
    where $R^{(l^{(m)},k)}:=N^{-1} \sum_{i=1}^N \prod_{a=1}^n (\sigma_i^{l^{(m)}_a})^{k_a}$ is the multioverlap with powers $k$ and replica indices $l^{(m)}$. By Lemma \ref{lem:mo_limit1} this implies
    \begin{equation*}
        \Etherm{R^{(k)}} - \frac{1}{t_N} \sum_{m=0}^{t_N-1} \int_0^1 \prod_{l=1}^n \bar \sigma^{(k_l)}(u,v,w_{l+nm}) \, dv \convdist 0\,.
    \end{equation*}
    As before, by the law of large numbers, we finally get
    \begin{equation*}
        \Etherm{R^{(k)}} \convdist \int_0^1 \cdots \int_0^1 \Big(\prod_{l=1}^n \bar \sigma^{(k_l)}(u,v,w_l) \, dw_l\Big) dv\,.
    \end{equation*}
\end{proof}

Obviously Theorem~\ref{prop:mo_thm_conc} is valid for the perturbed model (because the perturbed Hamiltonian $\hamil'$ remains concave). We have that under its hypotheses (for all $n\geq1$ and $k\in K_n$) $\noverlap$ converges weakly to $\Etherm{\noverlap}$; which by the unicity of the limit and Lemmas \ref{lem:mo_limit1} and \ref{lem:mo_limit2} implies that in the subsequential limit along $(N_j)_{j\ge 1}$,
\begin{equation*}
    R^{(k)}_\infty \eqdist \big\la{R^{(k)}_\infty}\big\ra\,.
\end{equation*}
Therefore $R^{(k)}_\infty$ is almost surely independent of $(w_l)_{l\le n}$. In the same way that Corollary~\ref{cor:asymp_spin_dist} is proven in Section~\ref{sec:proof_coro}, we can deduce from this fact that there exists in the above subsequential limit a simplified Aldous-Hoover representation $\sigma(u,v_i,x_{i,l})$ of the spin variables which is independent of $w_l$ and that is equivalent to the original one $\sigma(u,v_i,w_l,x_{i,l})$ for asymptotically describing the replicated system $(\sigma_{i}^l)_{i,l\ge 1}$. The equivalence between these two representations is in the sense that the joint distribution of the variables $(\sigma(u,v_i,x_{i,l}))_{i,l\ge 1}$ is the same as the one of $(\sigma(u,v_i,w_{l},x_{i,l}))_{i,l\ge 1}$. Slightly abusing notation we continue to denote with the same symbol $\sigma$ the new function associated with the simplified representation, and by $\bar{\sigma}^{(k)}$ the associated generalised magnetisation. This then proves the following corollary.
\begin{corollary}[Simplified Aldous-Hoover representation]\label{cor:thm_pure_state}
     Under the hypotheses of Theorem~\ref{prop:mo_thm_conc}, there exists some measurable function $\sigma:[0,1]^3 \to \interv$ such that along the converging subsequence $(N_j)_{j\ge 1}$, for every $i,l\geq1$, we have that
\begin{equation*}
    \sigma_i^l \convdist \sigma(u,v_i,x_{i,l})\,,
\end{equation*}
where (for every $i,l\geq1$) $u,v_i,x_{i,l}$ are i.i.d. uniform random variables on $[0,1]$. The generalised magnetisation associated to the simplified representation reads $$\bar \sigma^{(k)}(u,v_i):= \int_0^1 \sigma^k(u,v_i,x)\, dx\,.$$
\end{corollary}

This kind of behaviour in the large system limit is usually referred in mathematical physics literature as a \emph{thermal pure state} (for more details on this and the use of the Aldous-Hoover representation in spin glasses, see \cite{panchenko2014replica,panchenko2016structure,panchenko2013spin}).

\section{Proofs}\label{sec:proofs}


\subsection{Proof of Theorem~\ref{prop:mo_thm_conc}}

The result is stated for the regularised (non perturbed) system, but we prove it here for the perturbed one. Of course, one can simply take the perturbation strenght $t=0$ in \eqref{eq:poiss_pert} to recover the result for the regularised model.

Fix $n\geq 1$. By definition, conditionally on the disorder $(J,\pi,U)$ and for every fixed $\lambda$, the $n$-replicas system has density proportional to $\exp\sum_{l=1}^n \hamil'(\sigma^l\mid J,\pi,U,\lambda)$. Because $\hamil'(\sigma \mid J,\pi,U,\lambda)$ is a concave function on $\Sigma_N$ with, owing to the Gaussian regularisation, an Hessian upper bounded by $-\eps_N\mathbb{I}$, then $\sum_{l=1}^n \hamil'(\sigma^l\mid J,\pi,U,\lambda)$ is also concave on $\Sigma_N^n$ with Hessian upper bounded by $- \eps_N\mathbb{I}$.

Now, fix $k\in K_n$. Then the associated multioverlap $\noverlap$ is a function on $\Sigma_N^n$ with partial derivatives given (for every $i\in[N]$ and $j\in[n]$) by
\begin{equation*}
    \frac{\partial}{\partial \sigma_i^{l}}R^{(k)} = \frac{1}{N} k_l (\sigma_i^l )^{k_l-1} \prod_{m\neq l}^n (\sigma_i^m )^{k_m}\,.
\end{equation*}
Thus, the square of the norm of the gradient of $\noverlap$ is given by
\begin{equation*}
    \big\|\nabla \noverlap\big\|^2 = \frac{1}{N^2} \sum_{i=1}^N \sum_{l=1}^n k_l^2 (\sigma_i^j)^{2(k_l-1)} \prod_{m\neq j}^n (\sigma_i^m)^{2 k_m} \leq \frac{\sum_{l=1}^n k_l^2}{N}\,.
\end{equation*}
An application of Corollary \ref{cor:brascampstrong} then ends the proof of Theorem~\ref{prop:mo_thm_conc}.


\subsection{Proof of strong replica symmetry for the perturbed model [Proposition \ref{prop:mo_conc}]}

Throughout all the proofs of this subsection, the Poisson perturbation is assumed to be at full strength ($t=1$) and therefore the parameter $t$ is omitted from the notation. Let $$\theta_{I,j}^l := P_I(\sigma_{U^I_j}^l)\quad \mbox{as well as}\quad  E_I(\sigma^l) := \sum_{j=1}^{\pi_I} \theta_{I,j}^l \quad \mbox{for every $I\in\dyadic$}\,.$$ Recall that when the replica index is omitted, it means that the replica considered is the first one $\sigma=\sigma^1$, or $\theta_{I,j}=\theta_{I,j}^1$. The Poisson perturbation of the Hamiltonian may be expressed as
\begin{equation*}
    \hamil^{\rm poiss}(\sigma^l\mid \pi,U,\lambda) = - \sum_{I \in \dyadic} \lambda_I E_I(\sigma^l)\,.
\end{equation*}
A core inequality for our proof is a version of the Franz-de Sanctis identity in spin glasses \cite{de2009self}.
\begin{theorem}[Franz-de Sanctis type inequality]\label{thm:FdS_ineq}
    If $\hamil\in\mathcal{C}^2(\Sigma_N)$ is concave in the sense of \eqref{eq:concave} and obeys the relation \eqref{eq:spin_symm}, then we have that for every $n\geq1$, $I\in\mathcal{I}$, and $f_n:\Sigma_N^n \to \interv$ a function of replicas $(\sigma^l)_{l=1}^n$ of the perturbed system,
    \begin{equation*}
        \Elam{ \left| \Edis{\frac{\Etherm{ f_n  \theta_{I,1} e^{-\lambda_I \sum_{l=1}^n \theta_{I,1}^l} } }{\Etherm{e^{ - \lambda_I \theta_{I,1}}}^n} } - \Ethdis{f_n } \Edis{ \frac{\Etherm{\theta_{I,1} e^{-\lambda_I  \theta_{I,1}}}}{\Etherm{e^{-\lambda_I \theta_{I,1}}}}} \right|} \leq 5 \sqrt{\frac{(v'_N)^{1/2}+L (s_N)^{1/2}}{s_N}} + \sqrt{\frac{2}{s_N}}\,,
    \end{equation*}
    for some constant $L > 0$ (recall $v'_N$ was defined in \eqref{eq:def_v_N_pois}).
\end{theorem}

Note that the bound is uniform over $I\in\dyadic$ and $n\geq1$. To prove this, we will first obtain a bound for the mean absolute difference between the random variables $E_I(\sigma)$ and their means. The approach we follow is reminiscent of the derivation of the Ghirlanda-Guerra identities in spin glasses \cite{ghirlanda1998general}.
\begin{lemma}[Energy concentration]\label{lem:ener_conc}
    Suppose $\hamil$ obeys the hypotheses of Theorem \ref{thm:FdS_ineq}. Then there exists some constant $L > 0$ s.t., for every $I\in\dyadic$, we have that 
    \begin{equation*}
        \Elam{ \Ethdis{ \big| E_I(\sigma) -\Ethdis{ E_I(\sigma)} \big| } } \leq 5 \sqrt{((v'_N)^{1/2}+ L (s_N)^{1/2}) s_N} + \sqrt{2 s_N}\,.
    \end{equation*}
\end{lemma}
\begin{proof}
    In this proof we use the two alternative forms of the Poisson perturbation given in definitions \eqref{eq:poiss_pert} and \eqref{eq:poiss_pert2}. We will show the conclusion by bounding its \emph{thermal fluctuations} and \emph{disorder fluctuations}. That is, we will use that
    \begin{equation}\label{eq:sep_thm_dis}
        \begin{split}
            \Elam{ \Ethdis{ \big| E_I(\sigma) -\Ethdis{ E_I(\sigma)} \big| } } & \leq  \Elam{\Ethdis{\big| E_I(\sigma) -\Etherm{ E_I(\sigma)} \big|}}  + \Elam{\big| \Etherm{E_I(\sigma)} -\Ethdis{ E_I(\sigma)} \big|} 
        \end{split}
    \end{equation}
    and bound both terms on the right separately.
    
    Here and in the rest of the paper, whenever we interchange derivatives with expectations, we are either implicitly applying in a straightforward way Proposition A.1.2 of \cite{Talagrand2011spina} or dominated convergence. In order to bound the first term we exploit that
    \begin{equation*}
        \frac{d}{d\lambda_I} \Ethdis{E_I(\sigma)} = - \EThdis{\big(E_I(\sigma)-\Etherm{E_I(\sigma)}\big)^2}.
    \end{equation*}
    Which, integrating over $\lambda_I$, implies that
    \begin{equation*}
        \E_{\lambda_I}{\EThdis{(E_I(\sigma)-\Etherm{E_I(\sigma)})^2}} = - 2 \Ethdis{E_I(\sigma)} \big|_{\lambda=1/2}^{\lambda=1} \leq 2 \E{\pi_I} = 2 s_N\,,
    \end{equation*}
    where for the last inequality we used that the values of polynomials $P_I(\cdot)$ are by construction contained in $[0,1]$. Then, by Cauchy-Schwarz inequality we have that
    \begin{equation}\label{eq:bound_thml}
        \Elam{\Ethdis{\big| E_I(\sigma) -\Etherm{ E_I(\sigma)} \big|}} \leq \sqrt{2 s_N}\,.
    \end{equation}
    
    For the second term, we have (recall definition \eqref{Free_en_pert} for the free energy)
    \begin{equation*}
        \frac{d}{d\lambda_I} F'_N = - \Etherm{E_I(\sigma)} \quad \mbox{and} \quad \frac{d}{d\lambda_I} \E F'_N = - \Ethdis{E_I(\sigma)}.
    \end{equation*}
    The second derivatives are instead, as often in statistical mechanics, related to fluctuations:
    \begin{equation*}
        \frac{d^2}{d\lambda_I^2} F'_N = \ETherm{\big(E_I(\sigma)-\Etherm{E_I(\sigma)}\big)^2} \geq 0 \quad \mbox{and} \quad \frac{d^2}{d\lambda_I^2} \E F'_N = \EThdis{\big(E_I(\sigma)-\Etherm{E_I(\sigma)}\big)^2} \geq 0\,.
    \end{equation*}
    We then have that $F'_N$ and $\E F'_N$ are both convex functions of $\lambda_I$ and that $|dF'_N/d{\lambda_I} -d \,\E F'_N/d {\lambda_I} | = |\Etherm{E_I(\sigma)}-\Ethdis{E_I(\sigma)}|$. We can therefore use the following standard lemma for convex functions (for a proof, see Lemma 3.2 of \cite{panchenko2013sherrington}).
    
    \begin{lemma}[Bound for convex functions] \label{lem:der_conv}
        Let $G(x)$ and $g(x)$ be convex functions. Let $\delta > 0$ and define $C^{-}_\delta(x) := g'(x) - g'(x-\delta) \geq 0$ and $C^{+}_\delta(x) := g'(x+\delta) - g'(x)  \geq 0$. Then,
        \begin{align*}
        |G'(x) - g'(x)| \leq \delta^{-1} \sum_{u \in \{x-\delta,\, x,\, x+\delta\}} |G(u)-g(u)| + C^{+}_\delta(x) + C^{-}_\delta(x)\,.
        \end{align*}
    \end{lemma}
    
    Choosing $F_N'$ as $G$ and $\E F'_N$ as $g$, by the mean value theorem and the uniform upper bound for the mean of $E_I(\sigma)$, we have that both $C^-_{\delta_N}$ and $C^+_{\delta_N}$ are smaller or equal than $\delta_N s_N$. For the other three terms, let $x \in [1/2,1], \ u \in \{ x- \delta_N, x ,x +\delta_N \}$ (with $\delta_N<1/2$). By triangular and Jensen's inequalities we have (here the free energies are seen as functions of $\lambda_I$)
    \begin{equation}\label{eq:dmitry_split}
        \begin{split}
            \E{\big| F'_N(\lambda_I=u) - \E F'_N(u) \big|}  \leq &  \ \E{\big| F'_N(u)  - \E_\pi F'_N(u) \big|} + \E{\big| F'_N(u) - \E_J F'_N(u) \big|},
        \end{split}
    \end{equation}
    where $\E_J( \cdot )$ is the expectation w.r.t. the disorder $J$ only (i.e., the variables $\pi$ are fixed) and vice versa for $\E_\pi( \cdot )$ (Jensen's inequality was used as $\E{| F'_N(\lambda_I) - \E_J F'_N(\lambda_I) |}\ge \E{| \E_\pi F'_N(\lambda_I) - \E_\pi\E_J F'_N(\lambda_I)|}$). For the first term on the r.h.s. of \eqref{eq:dmitry_split} we have by Cauchy-Schwarz and Efron-Stein's variance bound (below $\pi'_{I,i}$ is an independent copy of $\pi_{I,i}$)
    \begin{equation*}
        \begin{split}
            \E{\big| F'_N(\lambda_I)  - \E_\pi F'_N(\lambda_I) \big|} & \leq \sqrt{\E \big(F'_N(\lambda_I)  - \E_\pi F'_N(\lambda_I) \big)^2} \\
            & \leq \sqrt{\frac{1}{2} \sum_{I\in\mathcal{I}} \sum_{i=1}^N \E \big(\ln\big\langle \exp{(- \lambda_I (\pi_{I,i}-\pi'_{I,i}) P_I(\sigma_{i}))}\big\rangle\big)^2} \\
            & \leq \sqrt{\frac{1}{2} \sum_{I\in\mathcal{I}}\sum_{i=1}^N 2^{-2\iota(I)-4m}m^2 \, \E(\pi_{I,i} - \pi'_{I,i})^2} \\
            & \leq L \sqrt{s_N}\,,
        \end{split}
    \end{equation*}
    where $m$ is the power of the monomial of maximum order entering $P_I$, and $L > 0$ is some constant. We used that $P_I$ is bounded by $2^{-\iota(I)-2m}m$, $\lambda_I\le 1$ and $\E(\pi_{I,i} - \pi'_{I,i})^2=2\,{\rm Var}\, \pi_{I,i}=2s_N/N$. For the second term in \eqref{eq:dmitry_split} we directly have that $\E{| F'_N(\lambda_I=u) - \E_J F'_N(\lambda_I=u)|} \leq (v'_N)^{1/2}$ by definition.
    
    The above bounds along with the previous convexity lemma and taking expectation over $\lambda$ yield
    \begin{equation}\label{eq:bound_dis}
       \Elam{\big| \Etherm{E_I(\sigma)} -\Ethdis{ E_I(\sigma)} \big|} \leq 3 \frac{(v'_N)^{1/2} + L (s_N)^{1/2}}{\delta_N} + 2 \delta_N s_N\,.
    \end{equation}
    Finally, we set $\delta_N = \sqrt{((v'_N)^{1/2}+L (s_N)^{1/2}) /s_N}$ and combine \eqref{eq:sep_thm_dis}, \eqref{eq:bound_thml} and \eqref{eq:bound_dis}.
\end{proof}

Using this lemma we can prove the Franz-de Sanctis type inequality for the perturbed system.

\begin{proof}[Proof of Theorem \ref{thm:FdS_ineq}]
    By Lemma \ref{lem:ener_conc}, for any $I\in\dyadic$,
    \begin{equation*}
    \Elam{\big|
    \E\langle {f_n  E_I(\sigma^1)}\rangle
    - \E\langle{ f_n }\rangle
    \E\langle{ E_I(\sigma^1) }\rangle  \big|} \leq 5 \sqrt{((v'_N)^{1/2}+ L (s_N)^{1/2}) s_N} + \sqrt{2 s_N}\,.
    \end{equation*}
    We are then left to prove that
    \begin{equation}\label{eq:equiv_FdS1}
        \E\big\langle {f_n  E_I(\sigma^1)}\big\rangle = s_N \Edis{\frac{\big\langle {f_n  \theta_{I,1} e^{- \lambda_I \sum_{l=1}^n \theta_{I,1}^l}}\big\rangle}{\big\langle{e^{- \lambda_I \theta_{I,1}}}\big\rangle^n} } \quad \mbox{and} \quad \E\big\langle { E_I(\sigma^1) }\big\rangle  = s_N \Edis{ \frac{\big\langle{\theta_{I,1} e^{-\lambda_I  \theta_{I,1}}}\big\rangle}{\big\langle{e^{-\lambda_I \theta_{I,1}}}\big\rangle}}\,.
    \end{equation}
    Conditioning on the value of $\pi_I$ (which appears in the thermal mean $\langle\,\cdot\, \rangle$ inside the Poisson perturbation) and taking the mean over $\pi_I$ explicitly we get that
    \begin{equation}\label{eq:pois_trick}
        \begin{split}
            \E\big\langle f_n  E_I(\sigma^1) \big\rangle & = \sum_{r = 0}^{+\infty} \frac{s_N^r}{r!}e^{-s_N} \E\big\langle f_n  \sum_{i = 1}^r \theta_{I,i} \,\big| \, \pi_I = r \big\rangle \\
            & = \sum_{r = 1}^{+\infty} \frac{s_N^r}{r!}e^{-s_N} \E\big\langle f_n  \sum_{i = 1}^r \theta_{I,i}\,\big| \, \pi_I = r \big\rangle \\
            & = \sum_{r = 1}^{+\infty} \frac{s_N^r}{r!}e^{-s_N} \,r\, \E\big\langle f_n \theta_{I,1} \,\big| \, \pi_I = r \big\rangle \\
            & = s_N \sum_{r = 1}^{+\infty} \frac{s_N^{r-1}}{(r-1)!}e^{-s_N} \E\big\langle f_n \theta_{I,1} \,\big|\, \pi_I = r \big\rangle \\
            & = s_N \sum_{m = 0}^{+\infty} \frac{s_N^{m}}{m!}e^{-s_N} \E\big\langle f_n \theta_{I,1} \,\big|\, \pi_I = m+1 \big\rangle\,.
        \end{split}
    \end{equation}
        Let us define an independent random variable $\tilde \pi_I \sim \mbox{Poiss}(s_N)$, and for $l\in[n]$ new variables
    \begin{equation}
        \tilde E_I(\sigma^l) := \sum_{j=1}^{\tilde \pi_I} \theta_{I,j+1}^l\,,
    \end{equation}
    as well as a new Hamiltonian $\hamil''$ according to
    \begin{equation}
        \hamil''(\sigma^l\mid J,\pi,\tilde\pi,U,\lambda) := \hamil(\sigma^l\mid J) + \hamil^{\rm gauss}(\sigma^l) - \sum_{I'\in\dyadic \backslash \{I\}} \lambda_{I'} E_{I'}(\sigma^l) - \lambda_I \tilde E_I(\sigma^l)\,.
    \end{equation}
    Writing explicitly the mean w.r.t. the perturbed Gibbs measure we get that
    \begin{equation*}
        \begin{split}
            \E\big\langle f_n  E_I(\sigma^1) \big\rangle & = \Edis{ \frac{ \int_{\Sigma_N^n} f_n  E_I(\sigma^1) e^{\sum_{l=1}^n \hamil'(\sigma^l\mid J,\pi,U,\lambda)} d\sigma^1 \ldots d\sigma^n } {(\int_{\Sigma_N} e^{\hamil'(\sigma\mid J,\pi,U,\lambda)} d\sigma)^{n}} } \\
            & = s_N \Edis{  \frac{ \int_{\Sigma_N^n} f_n  \theta_{I,1} e^{-\lambda_I \sum_{l=1}^n \theta_{I,1}^l} e^{\sum_{l=1}^n \hamil''(\sigma^l\mid J,\pi,\tilde\pi,U,\lambda)} d\sigma^1 \ldots d\sigma^n }{(\int_{\Sigma_N} e^{-\lambda_I\theta_{I,1}} e^{\hamil''(\sigma \mid J,\pi,\tilde\pi,U,\lambda)} d\sigma)^{n}}  } \\
            & = s_N \Edis{\frac{\big\langle f_n  \theta_{I,1} e^{-\lambda_I \sum_{l=1}^n \theta_{I,1}^l}\big\rangle}{\big\langle{e^{- \lambda_I \theta_{I,1}}\big\rangle^n}  }}\,,
        \end{split}
    \end{equation*}
    where in the second equality we used equation \eqref{eq:pois_trick}, and in the third one we multiplied and divided by $(\partf'')^n:=(\int d\sigma \exp \hamil''(\sigma \mid J,\pi,\tilde\pi,U,\lambda) )^n$ and then finally used that $\hamil'(\sigma^l\mid J,\pi,U,\lambda)$ equals $\hamil''(\sigma^l\mid J,\pi,\tilde\pi,U,\lambda)$ in distribution. This proves the first identity in \eqref{eq:equiv_FdS1}. The second relation is obtained from this one by taking $f_n = 1$.
\end{proof}

As special case we derive a decoupling lemma that will allow for the proof of our main theorem.

\begin{lemma}[Decoupling lemma]\label{lem:decouplig_lemma}
    Recall $\theta_{I,1} := P_I(\sigma_{U^I_1})$ and define $A_N$ as the right-hand side of the bound in Theorem~\ref{thm:FdS_ineq}. For all $I\in\dyadic$ and $(r_N)_{N\geq1}$ s.t. $r_N \to +\infty$,
    \begin{equation*}
        \Elam{ \left| \Edis{ \frac{\Etherm{\theta_{I,1} e^{-\lambda_I\theta_{I,1}} \theta_{I,2} e^{-\lambda_I\theta_{I,2}}}}{\Etherm{e^{-\lambda_I\theta_{I,1}}e^{-\lambda_I\theta_{I,2}}}} } - \bigg( \Edis{ \frac{\Etherm{\theta_{I,1} e^{-\lambda_I\theta_{I,1}}} }{\Etherm{e^{-\lambda_I\theta_{I,1}}}} } \bigg)^2  \right|} \leq 2e^4(1-e^{-2})^{r_N+1} + 2^{r_N+1} A_N\,.
    \end{equation*}
\end{lemma}
\begin{proof}
    To ease a bit the notation, let us define the random variables
    \begin{equation*}
        X := \big\langle\theta_{I,1} e^{-\lambda_I\theta_{I,1}} \theta_{I,2} e^{-\lambda_I\theta_{I,2}}\big\rangle\,, \quad V := \big\langle{e^{-\lambda_I\theta_{I,1}} e^{-\lambda_I\theta_{I,2}}\big\rangle}\,, \quad Y := \big\langle{\theta_{I,1} e^{-\lambda_I\theta_{I,1}}}\big\rangle, \quad\mbox{and}\quad W := \big\langle{e^{-\lambda_I\theta_{I,1}}}\big\rangle\,.
    \end{equation*}
    With it, we can rewrite the statement of the lemma as
    \begin{equation*}
        \Elam{\big| \Edis{(X/V)} - \big(\Edis{(Y/W)}\big)^2 \big| } \leq 2e^4(1-e^{-2})^{r_N+1} + 2^{r_N+1} A_N\,.
    \end{equation*}
    
    Let us write
    \begin{equation*}
         \frac{X}{V} =  \frac{ X / W } { V / W },
    \end{equation*}
    where the division in the denominator belongs a.s. to the interval $[e^{-2},1)$ (recall $0< P_I(x) \leq 1$ for all $x\in\interv$). Let $(p_r(\cdot))_{r\geq1}$ be the sequence of polynomials on $[e^{-2},1)$ defined by $$p_r(x) := \sum_{m=0}^r (1-x)^m = \sum_{m=0}^r \sum_{m'=0}^m\binom{m}{m'} (-x)^{m'}\,.$$ By the geometric series formula we have that this sequence converges (uniformly in $[e^{-2},1)$) to the function $1/x$. Furthermore, it is easy to check that $\sup_{x\in[e^{-2},1)} |p_r(x) - 1/x| \leq e^2(1-e^{-2})^{r+1}$ . We then have that

    \begin{equation*}
        \begin{split}
            \Elam{\big| \Edis{(X/V)} - \big(\Edis{(Y/W)}\big)^2 \big| } & \leq  \Elam{\big| \Edis{(X/V)} - \Edis{[(X/W) p_r(V/W)]} \big| } \\
            & \qquad + \Elam{\big|  \Edis{(Y/W)} \Edis{[Y p_r(W)]} -\big(\Edis{(Y/W)}\big)^2\big| } \\
            & \qquad\qquad + \Elam{\big| \Edis{[(X/W) p_r(V/W)]} - \Edis{(Y/W)} \Edis{[Y p_r(W)]}\big| }\,.
        \end{split}
    \end{equation*}
    Each of the first two terms on the right-hand side is smaller or equal to $e^4(1-e^{-2})^{r+1}$, because of the uniform convergence of $p_r(\cdot)$ and the fact that all the random variables involved and their divisions are explicitly bounded. 

    In order to bound the last term we will use Theorem~\ref{thm:FdS_ineq}. Instead of doing it directly, we use
    \begin{equation*}
        \begin{split}
            &\Elam{\big| \Edis{[(X/W) p_r(V/W)]} - \Edis{(Y/W)} \Edis{[Y p_r(W)]}\big| } \\
            & \ \ \ \ \ \ \ \ \ \ \ \ \ \ \ \ \ \ \ \ \ \ \ \ \ \ \ \ \ \ \ \ \ \ \  \leq \sum_{m=0}^r \sum_{m'=0}^m \binom{m}{m'} \Elam{\big| \Edis{[(X/W) (V/W)^{m'}]} - \Edis{(Y/W)} \Edis{[Y W^{m'}]}\big| }\,,
        \end{split}
    \end{equation*}
    where each of the terms on the right-hand side can be bounded using Theorem \ref{thm:FdS_ineq}. For this, note that $(X/W)(V/W)^{m'}$ can be rewritten using replicas as
    \begin{equation*}
        \frac{\big\langle{\theta_{I,1}  e^{-\lambda_I\sum_{l=1}^{m'+1} \theta_{I,1}^l} \theta_{I,2}   e^{-\lambda_I\sum_{l=1}^{m'+1} \theta_{I,2}^l}}\big\rangle}{\big\langle{e^{-\lambda_I \theta_{I,1}}}\big\rangle^{m'+1}} \,.
    \end{equation*}
    Using this and Theorem \ref{thm:FdS_ineq} choosing as bounded function $f_{m'+1} := \theta_{I,2} \exp(-\lambda_I \sum_{l=1}^{m'+1} \theta_{I,2}^l)$ yields
    \begin{equation*}
       \sum_{m=0}^r \sum_{m'=0}^m \binom{m}{m'} \Elam{\left| \Edis{[(X/W) (V/W)^{m'}]} - \Edis{(Y/W)} \Edis{[Y W^{m'}]}\right| } \leq 2^{r+1} A_N\,,
    \end{equation*}
    where we used that $\sum_{m=0}^r \sum_{m'=0}^m \binom{m}{m'} = \sum_{m=0}^r 2^{m}=2^{r+1}-1$. 

	Combining the bounds for the three terms and setting $r = r_N$ yields the result.
\end{proof}

\begin{proof}[Proof of Proposition \ref{prop:mo_conc}] 
    Recall that $A_N$ is the right-hand side of the bound in Theorem~\ref{thm:FdS_ineq}. It is easy to see that if $A_N = o_N(1)$, then one can always find a proper sequence $r_N$ s.t. the right-hand side of Lemma \ref{lem:decouplig_lemma} is also $o_N(1)$. Then, by hypothesis we have that the values of $(s_N)_{N\geq1}$ can be chosen appropriately so that the right-hand side of the inequality in Lemma \ref{lem:decouplig_lemma} is an $o_N(1)$ function of $N$. We then have as a direct consequence that as $N\to+\infty$ the following convergence in distribution over $\lambda$ holds for all $I\in\mathcal{I}$:
    \begin{equation}\label{eq:conv_variance}
        \Edis{ \frac{\Etherm{\theta_{I,1} e^{-\lambda_I\theta_{I,1}} \theta_{I,2} e^{-\lambda_I\theta_{I,2}}}}{\Etherm{e^{-\lambda_I\theta_{I,1}}e^{-\lambda_I\theta_{I,2}}}} } - \bigg( \Edis{ \frac{\Etherm{\theta_{I,1} e^{-\lambda_I\theta_{I,1}}} }{\Etherm{e^{-\lambda_I\theta_{I,1}}}} } \bigg)^2 \convdist 0\,.
    \end{equation}

    We will prove the theorem by contradiction. Assume there exists some subsequence $(N_j)_{j\geq1}$, $\delta > 0$, $n\geq1$, and $k \in K_n$, such that along the subsequence $(N_j)_{j\ge 1}$ it holds that
    \begin{equation}\label{eq:non_conc_R}
        \Elam{\EThdis{\big( R^{(k)} - \Ethdis{R^{(k)}} \big)^2 }} \geq \delta\,.
    \end{equation}
    Because the spin variables are tight, for any $\lambda\in [-1/2,1]^{\mathbb{N}}$ there exists a further subsequence of $(N_j)_{j\geq1}$ along which they converge to some asymptotic random variables. But, as discussed in Section \ref{sec:technical}, these limiting spin random variables are separately exchangeable. Then, there exists an Aldous-Hoover representation as in Theorem \ref{thm:AHrep}. Keep in mind that this representation $\sigma=\sigma_\lambda$ depends on the perturbation parameters $\lambda$ that will play a special role compared to the remaining randomness; yet we will keep this $\lambda$-dependence implicit in the rest of the argument. Furthermore, Lemmas \ref{lem:mo_limit1} and \ref{lem:mo_limit2} and Corollary \ref{cor:thm_pure_state} also hold for this limit. We can therefore direclty work with the simplified Aldous-Hoover representation $\sigma(u,v_i,x_{i,l})$ and associated generalised magnetisation $\bar \sigma^{(k)}(u,v_i)$ (that both depend implicitely on $\lambda$).
    
    In the rest of the proof we will freely interchange limits and expectations, which is valid because every random variable involved is bounded. Recall $\theta_{I,j}^l := P_I(\sigma_{U^I_j}^l)$. Define $$\bar \theta_{I}=\bar \theta_I(u,v_1,x_{1,1}) := P_I(\sigma(u,v_1,x_{1,1}))\,.$$ Here we fix the spin index to $1$ and so we exclude it from the notation, but of course the conclusions generalise to all other spins. Remember that when we omit the replica index, it is implicit we are referring to the first replica. Recall also that the thermal mean $\langle g((\sigma_i^l)_{i,l}) \rangle$ of a function of multiple replicas asymptotically converges, in the subsequential limit, to an average of this function $\int_0^1 \cdots\int_0^1 g((\sigma(u,v_i,x_{i,l}))_{i,l}) \prod_{i,l}dx_{i,l}$ over all the variables indexed by a replica, i.e., simply $(x_{i,l})$ when using the simplified representation. This can be seen from a slight generalisation of the proof of Lemma~\ref{lem:mo_limit2} in the Appendix (see also the Appendix in \cite{barbier2020strong}). By this convergence of the thermal mean and the fact that the ratio below can be expressed, using replicas, as a linear combinations of thermal means,
    \begin{align*}
        &\Edis{ \frac{\big\langle{\theta_{I,1} e^{-\lambda_I\theta_{I,1}} \theta_{I,2} e^{-\lambda_I\theta_{I,2}}}\big\rangle}{\big\langle{e^{-\lambda_I\theta_{I,1}}e^{-\lambda_I\theta_{I,2}}}\big\rangle} }  \\
        & \ \ \ \ \ \ \ \ \ \ \ \ \ \ \ \ \ \ \ \ \!=\! \sum_{m=0}^{k} \sum_{m'=0}^{m} \binom{m}{m'} (-1)^{m'} \Edis{ \big\langle{\theta_{I,1} e^{-\lambda_I\theta_{I,1}} \theta_{I,2} e^{-\lambda_I\theta_{I,2}} e^{-\lambda_I \sum_{l=2}^{m'+1} (\theta_{I,1}^l+\theta_{I,2}^l)}}\big\rangle} + o_k(1)
    \end{align*}
    we deduce using Corollary \ref{cor:thm_pure_state} that in this subsequential limit (setting $x=x_{1,1}$ and $v=v_1$) we almost surely have that
    \begin{equation*}
        \Edis{ \frac{\big\langle{\theta_{I,1} e^{-\lambda_I\theta_{I,1}} \theta_{I,2} e^{-\lambda_I\theta_{I,2}}}\big\rangle}{\big\langle{e^{-\lambda_I\theta_{I,1}}e^{-\lambda_I\theta_{I,2}}}\big\rangle} }\to \int_0^1 \bigg( \int_0^1 \frac{\int_0^1  \bar \theta_I e^{-\lambda_I \bar \theta_I}\, dx}{\int_0^1 e^{-\lambda_I \bar \theta_I} \, dx } dv \bigg)^2 du\,.
    \end{equation*}
    Similarly, we also have in this subsequential limit that
    \begin{equation*}
         \bigg( \Edis{ \frac{\Etherm{\theta_{I,1} e^{-\lambda_I\theta_{I,1}}} }{\Etherm{e^{-\lambda_I\theta_{I,1}}}} } \bigg)^2 \to \bigg( \int_0^1 \int_0^1 \frac{\int_0^1  \bar \theta_I e^{-\lambda_I \bar \theta_I}\, dx}{\int_0^1 e^{-\lambda_I \bar \theta_I} \, dx} dv \, du \bigg)^2\,.
    \end{equation*}
   	Then, if we define for each $I\in\mathcal{I}$ the random variable
    \begin{equation}\label{eq:rand_var_const}
        Y_{\lambda,I}(u) := \int_0^1 \frac{\int_0^1 \bar \theta_I e^{-\lambda_I \bar \theta_I} \, dx}{\int_0^1 e^{-\lambda_I \bar \theta_I} \, dx} \, dv\,,
    \end{equation}
    by the unicity of the limit, equation \eqref{eq:conv_variance} implies that the variance of $Y_{\lambda,I}(u)$ is $0$ and thus it is a.s. independent of $u$. Note that this conclusion is valid for every $I\in\mathcal{I}$. If (for $x\in\interv$) $P_I(x) = 2^{-\iota(I) - 2 m} \sum_{p=0}^{m-1} a_p (x+1)^p$, define new coefficients $\bar a_p := 2^{-\iota(I)- 2 m}\lambda_I a_p$. Then, because the sets $\mathcal{I}_m$ accumulate at $(0)$, taking (for each fixed $m\geq1$) the coefficients $(\bar a_p)_{p=0}^{m-1}$ as variables, the conclusion also holds for limits of these coefficients going to $0$ and derivatives w.r.t. them evaluated at $0$: by analyticity in $(\bar a_p)_{p=0}^{m-1}$ of the function $Y_{\lambda,I}(u) - \mathbb{E}Y_{\lambda,I}$ we have that $Y_{\lambda,I}(u)$ is a.s. constant for $\bar \theta_I$ defined by any polynomial $P_I$ with non-negative coefficients in a small enough neighbourhood of $0$. Because of this, when analysing some multioverlap $R^{(k)}=N^{-1} \sum_{i=1}^N (\sigma_i^1)^{k_1} \cdots (\sigma_i^n)^{k_n}$, we may consider a random variable of the form \eqref{eq:rand_var_const} but with a polynomial that only contains the powers $\{k_1,\dots,k_n\}$, which simplifies the notation.
    
    Using the fact that the derivatives cannot depend on $u$, we will contradict the hypothesis \eqref{eq:non_conc_R}. More specifically, we will prove by induction that, in this subsequential limit, every multioverlap has a self-averaging behaviour.
    
    Similarly to Section \ref{sec:settings}, we define, for every $n\geq1$, the sets $\tilde K_n := \{ k\in\N^n : k_1 \leq k_2 \leq \dots\leq k_n\}$ and $\tilde K := \bigcup_{n\ge 1} \tilde K_n$. As before, for every element $k:=(k_1,\dots,k_n)\in \tilde K$, we can associate to it the ($\lambda$-dependent) multioverlap asymptotic equivalent (recall Corollary~\ref{cor:thm_pure_state}):
    \begin{equation*}
        R^{(k)}_\infty=R^{(k)}_\infty(u) = \int_0^1 \prod_{l=1}^n \bar \sigma^{(k_l)}(u,v)\, dv\,.
    \end{equation*}
    
    We will give the set $\tilde K$ the lexicographic order $<_{\tilde K}$ defined for every pair of distinct elements $k:=(k_1,\dots,k_n)$, $k':=(k'_1,\dots,k'_{n'})\in \tilde K$ as per:
    \begin{enumerate}[(a)]
        \item when $n < n'$, then $k <_{\tilde K} k'$,
        \item in the case where $n = n'$, define $l_{\rm max} := \max\{l\in[n]: k_l \neq k'_l\}$. If $k_{l_{\rm max}} < k'_{l_{\rm max}}$, then $k <_{\tilde K} k'$.
    \end{enumerate}
    This is in fact a total order for the set $\tilde K$ and its smallest element is $(1)\in \tilde K$. Likewise, if we restrict it to the sets (for $n\geq1$) $\tilde K_n$, it is still a total order for them and the smallest element will be the $n$-dimensional vector of all ones $({1,\dots,1 })\in \tilde K_n$.
    
    We will prove by induction over this order that for each set $\tilde K_n$, every associated multioverlap is self-averaging. We will do this based on a second induction on the set index $n\geq1$. We mark explicitly the different steps of both inductions:
    
    \begin{enumerate}[(i)]
        \item First, we will see that the conclusion holds for $(1)\in \tilde K_1$. This is equivalent to saying that the magnetisation is self-averaging, i.e., it has an a.s. constant limit $R^{(1)}_\infty$. In this case we fix $\bar \theta_I = \bar a_1 (\sigma+1)$ (where $\sigma=\sigma(u,v,x)$ is the simplified Aldous-Hoover subsequential representation). We then have that for every $a_1 \in \mathcal{I}_1$, the expression (recall $\bar a_p := 2^{-\iota(I)-2m}\lambda_I a_p$)
        \begin{equation*}
            \int_0^1 \frac{\int_0^1 (\sigma+1) e^{-\bar a_1 (\sigma+1)} \, dx}{\int_0^1 e^{-\bar a_1 (\sigma+1)} \, dx} \, dv
        \end{equation*}
        is a.s. constant. By taking the limit $\bar a_1\to0$, we have that $\int_0^1\int_0^1 (\sigma+1) \, dx \, dv$ is also a.s. constant. And because by Lemma \ref{lem:mo_limit1} and Corollary \ref{cor:thm_pure_state} we have $R^{(1)}_\infty(u) = \int_0^1 \bar \sigma^{(1)}(u,v) \, dv$, therefore $R^{(1)}_\infty$ is also a.s. constant.
        \item We will now see that given $k\in \tilde K_1$, if for every $k'\in \tilde K_1$ s.t. $k' <_{\tilde K} k$ we have that $R^{(k')}_\infty$ is self-averaging, then $R^{(k)}_\infty$ is also. For this, set $\bar \theta_I = \bar a_1 (\sigma+1)^k$. We then have that
        \begin{equation*}
            \int_0^1 \frac{\int_0^1 (\sigma+1)^k e^{-\bar a_1 (\sigma+1)^k} \, dx}{\int_0^1 e^{-\bar a_1 (\sigma+1)^k} \, dx} \, dv
        \end{equation*}
        is a.s. constant. Taking limit $\bar a_1\to0$ we get that also is $\int_0^1\int_0^1 (\sigma+1)^k \, dx \, dv$. By Newton's binomial we have that
        \begin{equation*}
            \int_0^1\int_0^1 (\sigma+1)^k \, dx \, dv = \int_0^1 \bar \sigma^{(k)} \, dv + \sum_{i=0}^{k-1} \binom{k}{i} \int_0^1 \bar \sigma^{(i)} \, dv\,.
        \end{equation*}
        Owing to the fact that (for every $i\in[k-1]$) by Lemma \ref{lem:mo_limit1} $R^{(i)}_\infty = \int_0^1 \bar \sigma^{(i)} \, dv$, we have that by the induction hypothesis the sum on the second term is a.s. constant. Which implies that also is the first term, as we wanted to prove.
        \item In this step we will see that if for every $n'<n$ the conclusion holds for every element of $\tilde K_{n'}$, then it is also true for the smallest element of $\tilde K_n$ (that is, for $k_{\rm min} := (1,\dots,1)\in \tilde K_n$). We fix $\bar \theta_I = \bar a_1 (\sigma+1)$. We know that
        \begin{equation*}
            \frac{\partial^{n-1}}{\partial \bar a_1^{n-1}}  \int_0^1 \frac{\int_0^1 (\sigma+1) e^{-\bar a_1 (\sigma+1)} \, dx}{ \int_0^1 e^{-\bar a_1 (\sigma+1)} \, dx} \, dv\,\Bigg|_{\bar a_1=0}
        \end{equation*}
        is a.s. a constant. It is easy to see that the only term depending on $n$ replicas $(\sigma(u,v,x_{1,l}))_{l\le n}$ comes from deriving the denominator $n-1$ times and is proportional to
        \begin{equation*}
            \int_0^1 \frac{\int_0^1\cdots\int_0^1 \prod_{l=1}^n (\sigma(u,v,x_{1,l})+1)  e^{-\bar a_1 \sum_{l=1}^n (\sigma(u,v,x_{1,l})+1)} \,dx_{1,l}}{\big(\int_0^1 e^{-\bar a_1 (\sigma(u,v,x)+1)} \, dx\big)^n} \, dv\,.
        \end{equation*}
        Taking limit $\bar a_1\to0$, all the terms that depend on $n-1$ replicas or less can be rewritten as combinations of multioverlap limits of $n-1$ or less replicas, which are a.s. constant by induction hypothesis. Then, we conclude that $\int_0^1 (\bar \sigma^{(1)}+1)^n dv$ must be a.s. constant. Again, if we expand the power inside this last integral, we get that the only term depending on $n$ replicas is $\int_0^1 (\bar \sigma^{(1)})^n dv = R^{(k_{\rm min})}_\infty$. Then we have that this term is a.s. constant, which proves the conclusion of this step.
        \item Finally, we prove that for each $k=(k_1,\dots,k_n)\in \tilde K$, if for every $k'\in \tilde K$ such that $k' <_{\tilde K} k$ the associated multioverlap concentrates, then so is the one associated to $k$. Set $\bar \theta_I = \sum_{p=1}^{n} \bar a_p (\sigma+1)^{k_p}$. By the fact that $Y_{\lambda,I}$ is a.s. constant we have that
        \begin{equation*}
            \frac{\partial^n}{\partial \bar a_{1}\cdots\partial \bar a_{n}} \int_0^1 \frac{\int_0^1 \bar \theta_I e^{- \bar \theta_I} \, dx}{\int_0^1 e^{- \bar \theta_I} \, dx}\, dv\,\Bigg|_{\bar a_{1},\dots,\bar a_{n}=0}
        \end{equation*}
        is also a.s. constant. Again, all the terms that depend on $n-1$ replicas or less can be rewritten as combinations of multioverlap limits of $n-1$ or less replicas, which are a.s. constant by induction hypothesis. And there will only be $n$ terms that depend on $n$ replicas, all of them with the same sign and proportional to
        \begin{equation*}
            \int_0^1\cdots\int_0^1 \Big(\prod_{l=1}^n (\sigma(u,v,x_{1,l})+1)^{k_l}  \,dx_{1,l}\Big)\, dv\,.
        \end{equation*}
        These terms are all obtained by deriving one time the factor $\bar \theta_I$ on the numerator and $n-1$ times the denominator (observe that the terms that do not involve a derivative of the factor $\bar \theta_I$ in the numerator become $0$ when evaluated at $\bar a_1=\dots=\bar a_n=0$). This last integral is therefore a.s. constant. By expanding and distributing the products in this integral, we find that it can be rewritten as
        \begin{equation*}
            \int_0^1\cdots\int_0^1 \Big(\prod_{l=1}^n (\sigma(u,v,x_{1,l})+1)^{k_l}  \, dx_{1,l}\Big)  dv = R^{(k)}_\infty + \sum_{k'<_{\tilde K} k} A_{k'} R^{(k')}_\infty + 1\,,
        \end{equation*}
        where the coefficients $A_{k'}$ are all positive integers. By induction hypothesis, the contribution $\sum_{k'<_{\tilde K} k} A_{k'} R^{(k')}_\infty $ is a.s. constant. This then proves that $R^{(k)}_\infty$ is a.s. independent of $u$ (but still depends on $\lambda$).
    \end{enumerate}
    
    By this induction, we conclude that for every $k\in K$, $(R^{(k)} -\mathbb{E}\langle R^{(k)}\rangle ) \convdist 0$ in the subsequential limit. Because the multioverlaps are a.s. bounded, this implies that for every $k\in K$, $$\big(R^{(k)} - \mathbb{E}\big\langle R^{(k)}\big\rangle\big)\xrightarrow{L^2} 0\,.$$ This contradicts the assumption \eqref{eq:non_conc_R}, which proves that there exists no such subsequence $(N_j)_{j\ge 1}$ of system sizes along which \eqref{eq:non_conc_R} holds. This finishes the proof.
\end{proof}

We have finished the proof of the main intermediate result. We now show how to obtain the results for the original (non perturbed) model from the ones we just proved for the perturbed one.


\subsection{Proof of Proposition \ref{prop:thm_mean_dif}}
    We consider the perturbed model \eqref{pert_mode} with perturbation strength parameter $t$. Recall that we denote by $\langle \,\cdot\, \rangle_t$ the expectation w.r.t. the Gibbs measure associated with this perturbed model. Recall also $E_I(\sigma^l) := \sum_{i=1}^{N} \pi_{I,i} P_I(\sigma^l_{i})$. For every $t\in [0, 1]$,
    \begin{equation*}
        \Big| \frac{d}{d t} \Etherm{R^{(k)}}_t \Big| =  \Big| \sum_{l=1}^n \Big\langle \big(R^{(k)}-\Etherm{R^{(k)}}_t\big) \sum_{I\in\mathcal{I}} \lambda_I \big( E_I(\sigma^l) - \big\langle {E_I}\big\rangle_t\big)\Big\rangle_t \Big|\,.
    \end{equation*}
    By Cauchy-Schwarz we get that
    \begin{equation*}
        \Big( \frac{d}{d t} \Etherm{R^{(k)}}_{t} \Big)^2 \leq n^2 \ETherm{\big(R^{(k)}-\Etherm{R^{(k)}}_t\big)^2}_t \Big\langle\big[\sum_{I\in\mathcal{I}} \lambda_I \big( E_I(\sigma) - \big\langle {E_I}\big\rangle_t\big)\big]^2\Big\rangle_t\,.
    \end{equation*}
    We will now bound both thermal variances on the r.h.s. by means of Brascamp-Lieb's inequality.
    
    For the multioverlap fluctuations, this is given by Proposition \ref{prop:mo_thm_conc}, that applies as the perturbed Hamiltonian is concave for any $t\in[0,1]$, and which states that
    \begin{equation*}
        \ETherm{\big(R^{(k)}-\Etherm{R^{(k)}}_t\big)^2}_t \leq \frac{\norm{k}^2}{N \varepsilon_N}\,,
    \end{equation*}
    where $\norm{k}^2 := \sum_{i=1}^{|k|} k_i^2$. For the thermal variance of $\sum_{I\in\mathcal{I}} \lambda_I E_I$ we first compute the square of the norm of its gradient as a function of $\sigma$. It is easy to see that for every $i \in [N]$,
    \begin{equation*}
        \frac{\partial}{\partial \sigma_i} \sum_{I\in\mathcal{I}} \lambda_I E_I = \sum_{I\in\mathcal{I}} \lambda_I \pi_{I,i} P'_I(\sigma_i) =  \sum_{I\in\mathcal{I}} \lambda_I \pi_{I,i} 2^{-\iota(I)- 2 m} \sum_{p=1}^{m-1} p a_p (\sigma_i+1)^{p-1} \leq \sum_{I\in\mathcal{I}} C_I \pi_{I,i} \, ,
    \end{equation*}
    where $C_I := 2^{-\iota(I)-2m-1} (m-1)^2$ (recall the Poisson perturbation is summable, so exchanging derivatives and the summation can be done safely). Now, if we define $X_i := \sum_{I\in\mathcal{I}} C_I \pi_{I,i}$, we have that $\norm{\nabla E_I}^2 \leq \sum_{i=1}^N X_i^2$. By Corollary \ref{cor:brascampstrong} we get that
    \begin{equation*}
       \Big\langle\big[\sum_{I\in\mathcal{I}} \lambda_I \big( E_I(\sigma) - \big\langle {E_I}\big\rangle_t\big)\big]^2\Big\rangle_t \leq \frac{1}{\varepsilon_N}\sum_{i=1}^N X_i^2\,.
    \end{equation*}

    All this put together shows that a.s.
    \begin{equation*}
        \Big| \frac{d}{d t} \Etherm{R^{(k)}}_{t} \Big| \leq n \norm{k} \sqrt{\frac{\sum_{i=1}^N X_i^2}{N \varepsilon^2_N}}\,.
    \end{equation*}
    We now set $\varepsilon_N$ to be equal to $(s_N/N)^{1/3}$, which we can do because the only constraints on this sequence are that it should be $o(1)$ and $N\varepsilon_N\to+\infty$. Then, we get that
    \begin{equation*}
        \E \Big| \Etherm{R^{(k)}}_{t=1} - \Etherm{R^{(k)}}_0 \Big| \le  \E \sup_{t\in[0,1]} \Big| \frac{d}{dt} \Etherm{R^{(k)}}_{t}\Big| \leq C(k) \Big(\frac{s_N}{N}\Big)^{1/6}(1+o_N(1)) = o_N(1)\,.
    \end{equation*}
    For the inequality we used the concave version of Jensen's inequality, that $\pi_{I,i}$ is Poisson with mean $s_N/N\to 0_+$, and $C(k) := n \norm{k} \sqrt{\sum_{I\in\mathcal{I}} C_I^2} < +\infty$. \qed

\subsection{Proof of strong replica symmetry for the original model [Theorem \ref{thm:mo_original_conc}]}\label{sec:proof_orig_conc}
We now extend the concentration result of Proposition \ref{prop:mo_conc} to the system without the Poisson perturbation ($t=0$) by the fact that the values of the thermal and disorder means of multioverlaps, with and without the perturbation, differ by a vanishing quantity.

Define $R_0^{(k)}$ as a multioverlap fluctuating according to the non perturbed measure, so $\langle R_0^{(k)} \rangle:=\langle R^{(k)}\rangle_{t=0}$. Instead $R^{(k)}$ fluctuates according to the perturbed measure, so $\langle R^{(k)} \rangle:=\langle R^{(k)}\rangle_{t=1}$. From Proposition~\ref{prop:thm_mean_dif} we immediately obtain that under the conditions of Proposition \ref{prop:mo_conc}, for every $k\in K$, we almost surely have that
\begin{equation}
    \Big| \E\big\langle R^{(k)}\big\rangle - \E\big\langle R^{(k)}_0\big\rangle \Big| \leq C(k) \Big(\frac{s_N}{N}\Big)^{1/6}(1+o_N(1))\,.\label{to_use}
\end{equation}
This can be used to prove that the multioverlaps for the original model without the Poisson perturbation (i.e., for $t = 0$) also concentrate. As usual, for a r.v. $X$ we will write $\norm{X}_2 := \sqrt{\Ethdis{X^2}}$. 

By the triangular inequality we have that (at $t=0$ the system does not depend on $\lambda$)
    \begin{equation*}
        \begin{split}
            \big\|{\noverlap_0 - \E\big\langle R_0^{(k)} \big\rangle} \big\|_2 & \leq \big\|{\noverlap_0 - \big\langle R_0^{(k)} \big\rangle} \big\|_2 + \Elam \norm{\Etherm{\noverlap} - \Ethdis{\noverlap}}_2 \\
            & \ \ \ \ \ \ \ \ \ + \Elam \big\|{\Etherm{\noverlap} - \big\langle R_0^{(k)} \big\rangle}\big\|_2 + \Elam\big\|{\Ethdis{\noverlap} - \E\big\langle R_0^{(k)} \big\rangle}\big\|_2\,.
        \end{split}
    \end{equation*}
The first term goes to zero by Theorem~\ref{prop:mo_thm_conc}, the second one by concave Jensen's inequality and Proposition \ref{prop:mo_thm_conc} and Proposition \ref{prop:mo_conc}, and the last two by Proposition~\ref{prop:thm_mean_dif} and \eqref{to_use}. \qed


\subsection{Proof of Corollaries \ref{cor:asymp_indep_thm}, \ref{cor:asymp_spin_dist} and \ref{cor:asymp_indep_disthm}}\label{sec:proof_coro}

As a direct consequence of Theorem \ref{thm:mo_original_conc}, the proofs of the corollaries present in Section \ref{sec:results} follow.

\begin{proof}[Proof of Corollary \ref{cor:asymp_spin_dist}]
    The fact that the multioverlaps concentrate (by Theorem \ref{thm:mo_original_conc}) means that along every subsequence of system sizes s.t. the spin variables $(\sigma_i^l)_{i,l\geq1}$ converge in distribution, for all $k\in K$, the associated asymptotic multioverlaps $R^{(k)}_\infty(u,w_1,\dots,w_n)$ are a.s. constant. This implies that there exist fixed values $u_0\in[0,1]$ and $(w_{0,l})_{l\geq1}\in[0,1]^\N$ s.t. (for all $k=(k_1,\dots,k_n)\in K$) $R^{(k)}_\infty(u,w_1,\dots,w_n) = R^{(k)}_\infty(u_0,w_{0,1},\dots,w_{0,n})$ almost surely. 

    We now prove that we may take the Aldous-Hoover limit to be given (for the spin $\sigma_i^l$, with $i,l\geq1$) by a new function $\tilde \sigma(v_i,x_{i,l}) := \sigma(u_0,v_i,w_{0,l},x_{i,l})$. Indeed, by doing so, the values of the means of all the multioverlaps and their products remain unchanged. This implies that the joint distribution of the variables $(\tilde \sigma(v_i,x_{i,l}))_{i,l\ge 1}$ is the same as the one of $(\sigma(u,v_i,w_{l},x_{i,l}))_{i,l\ge 1}$. To see this, note that any generic joint moment of spins given by (for $\mathcal{C} \in \N^n \times \N^n$)
    \begin{equation*}
        m_\mathcal{C} := \E \Big\langle \prod_{(i,l)\in\mathcal{C}} \sigma_i^l  \Big\rangle\,,
    \end{equation*}
    can be straightforwardly expressed as the mean of a product of multioverlaps. By what we said previously these moments are asymptotically the same when computed using representation $\sigma$ or the simplified one $\tilde \sigma$: the joint moments generating functions of spin variables distributed according to both Aldous-Hoover limits are the same. This proves that the joint distributions of both limits are the same. Therefore, they are equivalent.
    
    The conclusion of the Corollary \ref{cor:asymp_spin_dist} is obtained by re-expressing the fact that the distributional limit of $(\sigma_i^l)_{i,l\geq1}$ along this convergent subsequence is $(\tilde \sigma(v_i,x_{i,l}))_{i,l\geq1}$ (with $(v_i)_{i\geq1}$ and $(x_{i,l})_{i,l\geq1}$ all i.i.d. uniform in $[0,1]$) in terms of random measures.
\end{proof}

\begin{proof}[Proof of Corollaries \ref{cor:asymp_indep_thm} and \ref{cor:asymp_indep_disthm}]
    The proof of Corollary \ref{cor:asymp_indep_disthm} requires the independence of the asymptotic Aldous-Hoover representation on both the variables $u$ and $(w_l)_{l\geq1}$. Corollary \ref{cor:asymp_indep_thm} is just a weaker version of this result derived in an analogous manner but for the case in which the Aldous-Hoover representation does not depend on $(w_l)_{l\geq1}$ but may still have a dependence on $u$. Its proof is thus omitted.

    Let $\{h_i\}_{i\in[k]}$ be as in the statement of the corollary and $(N_j)_{j\ge 1}$ be a subsequence of $N$. Because the spin variables are tight, we have that there exists some subsubsequence $(N_{j_m})_{m\ge 1}$ along which the spins $\sigma_1,\dots,\sigma_k$ (and their replicas) converge jointly in distribution, with asymptotic Aldous-Hoover representation $\sigma(u,v,w,x)$. By the independence of $\sigma(u,v,w,x)$ on $u$ and $w$ (consequence of Corollary \ref{cor:asymp_spin_dist}), we have that along this subsubsequence strong decoupling holds:
    \begin{equation}
        \E \Big\langle{\prod_{j=1}^k h_j(\sigma_{j})}\Big\rangle - \prod_{j=1}^k \E\Big\langle{h_j(\sigma_{j})}\Big\rangle \rightarrow 0\,.\label{meanh-hmean}
    \end{equation}
   We have then proved that every subsequence has a subsubsequence along which strong decoupling holds. Because the limit is always the same, we have the convergence along the entire sequence $N$. This means that as $N\to+\infty$ \eqref{meanh-hmean} holds.
\end{proof}

\section*{Acknowledgements}
J.B. would like to thank Nicolas Macris and Marc M\'ezard for discussions and their support. D.P. was partially supported by NSERC and Simons Fellowship. M.S. would like to thank Matthieu Jonckheere for all the academic and personal help throughout the last years.


\end{document}